\documentstyle[amsfonts, txfonts, amssymb, latexsym, 12pt]{amsart}

\newtheorem{theorem}{Theorem}[section]
\newtheorem{lemma}[theorem]{Lemma}

\def\Z{{\mbox{\rm\kern.25em
\vrule width.03em height0.57ex depth0ex
\kern.033em
\vrule width.03em height1.52ex depth-0.96ex \kern-.338em Z}}}
\def\R{{\mbox{\rm I\kern-.22em R}}}
\def\N{{\mbox{\rm I\kern-.22em N}}}

\def\supp{{\rm supp}}

\def\max{{\rm max}}
\def\sgn{{\rm sgn}}

\def\k{{\kappa}}

\def\S{{\cal{S}}}
\def\J{{\cal{J}}}

\def\dist{{\rm dist}}

\def\111{\gamma}

\def\be#1{\begin{equation}\label{#1}}
\def\bas{\begin{align*}}
\def\eas{\end{align*}}
\def\bi{\begin{itemize}}
\def\ei{\end{itemize}}
\newenvironment{proof}{\noindent {\bf Proof} }{\endprf\par}
\def \endprf{\hfill  {\vrule height6pt width6pt depth0pt}\medskip}
\def\emph#1{{\it #1}}

\title[Calder\'{o}n commutators and Cauchy integral II]{Calder\'{o}n commutators and the Cauchy integral on Lipschitz curves revisited II. Cauchy integral and generalizations}

\author{Camil Muscalu}
\address{Department of Mathematics, Cornell University, Ithaca, NY 14853}
\email{camil@@math.cornell.edu}

\begin{document}

\begin{abstract}
This article is the second in a series of three papers, whose scope is to give new proofs to the well known theorems of Calder\'{o}n, Coifman, McIntosh and Meyer \cite{calderon}, \cite{meyerc}, \cite{cmm}. Here we treat the case of
the Cauchy integral on Lipschitz curves and some of its generalizations.
\end{abstract}

\maketitle

\section{Introduction}

This paper is a continuation of \cite{camil} and it is the second of a three papers sequel.
Let $A$ be a Lipschitz function on the real line $\R$. It defines a Lipschitz curve $\Gamma$ in the complex plane, by the parametrization $x\rightarrow x+iA(x)$. The Cauchy integral associated with this curve is the 
singular integral operator $C_{\Gamma}$ given by 

\begin{equation}\label{1}
C_{\Gamma}f(x) := p.v. \int_{\R}\frac{f(y)}{(x-y) + i (A(x) - A(y))} dy.
\end{equation}

The goal of this article is to give a new proof to the well known theorem of Coifman, McIntosh and Meyer \cite{cmm}, which says that $C_{\Gamma}$ extends naturally as a linear bounded operator from
$L^p$ into $L^p$ for any $1<p<\infty$. Moreover, the method of proof will also allow us to obtain several new generalizations of this important theorem, which will be described in the last section of the paper.

As it is known, standard arguments reduce this problem to the problem of proving polynomial bounds for the associated Calder\'{o}n commutators, defined by

\begin{equation}\label{2}
C_d f(x):= p.v. \int_{\R} \frac{(A(x)-A(y))^d}{(x-y)^{d+1}} f(y) dy.
\end{equation}
More precisely, it is enough to prove that

\begin{equation}\label{3}
\|C_d f \|_p \leq C(d) \cdot C(p)\cdot \|f\|_p\cdot \|A'\|_{\infty}^d
\end{equation}
for any $f\in L^p$, where $C(d)$ grows at most polynomially in $d$.

Simple and standard calculations, similar to the ones in \cite{camil}, show that for $a:= A'$ and $f$ Schwartz functions, (\ref{2}) exists and can be rewritten as

\begin{equation}\label{4}
\int_{\R^{d+1}}
\left[\int_{[0,1]^d} \sgn(\xi +\alpha_1\xi_1 + ... + \alpha_d\xi_d) d\alpha_1\cdots d\alpha_d\right]
\widehat{f}(\xi)
\widehat{a}(\xi_1) ...
\widehat{a}(\xi_d)
e^{2\pi i x(\xi +\xi_1 + ... +\xi_d)}
d\xi
d\xi_1 ...
d\xi_d.
\end{equation}
As a consequence, $C_d$ can be seen as a $(d+1)$- linear operator. More specifically, it is given by the map

$$(f, g_1, ..., g_d) \rightarrow $$

$$\int_{\R^{d+1}}
\left[\int_{[0,1]^d} \sgn(\xi +\alpha_1\xi_1 + ... + \alpha_d\xi_d) d\alpha_1\cdots d\alpha_d\right]
\widehat{f}(\xi)
\widehat{g_1}(\xi_1) ...
\widehat{g_d}(\xi_d)
e^{2\pi i x(\xi +\xi_1 + ... +\xi_d)}
d\xi
d\xi_1 ...
d\xi_d.
$$

However, since its symbol 

\begin{equation}\label{symbol}
m_d(\xi, \xi_1, ..., \xi_d) : = \int_{[0,1]^d}\sgn(\xi +\alpha_1\xi_1 + ... + \alpha_d\xi_d)d\alpha_1 ... d\alpha_d
\end{equation}
 is not a classical
Marcinkiewicz-Mihlin-H\"{o}rmander symbol, there are no estimates for $C_d$ that can be easily passed to the multilinear theorem of Coifman and Meyer \cite{meyerc} and this is why proving (\ref{3}) even without polynomial bounds, is a more delicate
problem than an estimate on paraproducts. In \cite{camil}, we gave a new proof of (\ref{3}) in the particular case of the first Calder\'{o}n commutator $C_1$. The proof was based on the observation
that even though $m_1(\xi,\xi_1)$ is not a classical symbol, when {\it smoothly} restricted to Whitney squares (with respect to the origin) the Fourier coefficients of the corresponding functions decay at least quadratically. This fact, together with
the logarithmical bounds for the shifted Hardy-Littlewood maximal functions and Littlewood-Paley square functions (also proved in \cite{camil}), were enough to reduce the problem to a setting where the method developed 
in \cite{mptt:biparameter} and \cite{mptt:multiparameter} could be applied. It is on the other hand not difficult to realize, that even if one would assume that such quadratic estimates hold true in the general case of $m_d(\xi, \xi_1, ..., \xi_d)$, 
these observations alone would not be enough
to obtain (\ref{3}), since then one would end up summing $O(d)$ power series, which would finally generate an exponential upper bound of type $C^d$. The main new idea to obtain the desired polynomial bounds in (\ref{3}), is to realize that
instead of treating $m_d$ as a {\it whole multiplier} of $d+1$ variables, one can see it as being a multiple average of various {\it $m_1$ type} multipliers. In other words, throughout this paper, we will never need to go beyond the understanding
of the symbol of the first Calder\'{o}n commutator, to be able to obtain polynomial bounds for all the other commutators. This may seem surprising at the first glance, but could also be seen as an {\it explanation} of a somewhat similar observation of 
Verdera \cite{verdera}
who showed that in a certain sense, the Cauchy integral {\it is dominated} by the first Calder\'{o}n commutator.

Now, coming back to (\ref{3}), we will prove the following

\begin{theorem}\label{th1}
Let $1<p_1, ..., p_{d+1}\leq\infty$ and $1\leq p < \infty$ be so that $1/p_1 + ... +1/p_{d+1} = 1/p$. Denote by $l$ the number of indices $i$ for which $p_i \neq \infty$. Then, $C_d$ extends naturally as a $(d+1)$- linear operator
bounded from $L^{p_1}\times ... \times L^{p_{d+1}} \rightarrow L^p$ with an operatorial bound of type

\begin{equation}
C(d)\cdot C(l)\cdot C(p_1)\cdot ... \cdot C(p_{d+1})
\end{equation}
where $C(d)$ grows at most polynomially in $d$ and $C(p_i) = 1$ as long as $p_i = \infty$ for $1\leq i\leq d+1$.
\end{theorem}

Assuming for a moment Theorem \ref{th1}, we see immediately that our desired estimate (\ref{3}) follows from it by taking $p_1=p$ and
$p_2 = ... =p_{d+1} =\infty$.

Let us remember now that since $C_d$ is a $(d+1)$- linear operator, it has $(d+1)$ natural adjoints. To define them, recall the definition of the associated $(d+2)$- linear form $\Lambda_d$ given by

\begin{equation}\label{7}
\int_{\R}C_d(f_1, ..., f_{d+1})(x) f_{d+2}(x) d x = \Lambda_d (f_1, ..., f_{d+2}).
\end{equation}
Then, for every $1\leq i\leq d+1$ one defines $C_d^{\ast i}$ by

\begin{equation}\label{8}
\int_{\R}C_d^{\ast i}(f_1, ..., f_{i-1}, f_{i+1}, ..., f_{d+2})(x) f_i(x) dx = \Lambda_d (f_1, ..., f_{d+2}).
\end{equation}
For symmetry, we also use the notation $C_d = C_d^{\ast d+2}$.

To prove Theorem \ref{th1}, we will show that for every $1\leq i\leq d+2$ and for every $\phi_1, ..., \phi_{d+1}$ Schwartz functions, one has

\begin{equation}\label{6}
\|C^{\ast i}_d(\phi_1, ..., \phi_{d+1})\|_p \leq C(d)\cdot C(l)\cdot C(p_1)\cdot ... \cdot C(p_{d+1})\cdot \|\phi_1\|_{p_1}\cdot ... \cdot \|\phi_{p_{d+1}}\|_{d+1},
\end{equation}
where $p_j$ for $1\leq j \leq d+1$ and $p$ are as before.
From this one can immediately extend $C^{\ast i}_d$ by density on an arbitrary product of $L^{p_j}$ spaces, as long as $1<p_j<\infty$ and on $\overline{\S^{\infty}}$ spaces
(the closure of the family of Schwartz functions in $L^{\infty}$) in the case when $p_j = \infty$.

In the next section we will explain how one can then use duality arguments, to define $C^{\ast i}_d$ even further, to generic products of $L^{p_j}$ and $L^{\infty}$ spaces. These duality arguments 
will also clarify
the necessity of proving the wider range of estimates which appear in Theorem \ref{th1} and (\ref{6}) for $C_d$ and its adjoints.

{\bf Acknowledgements:} The author wishes to thank the referees for their careful corrections, which helped to improve the presentation. 
The present work has been partially supported by the NSF.

\section{Duality and the extension from $\overline{\S^{\infty}}$ to $L^{\infty}$}

For any $\# = 0, 1, ..., d$ denote by $S(\#)$ the statement that the inequalities (\ref{6}) for $C_d$ and all its adjoints, extend naturally to the situation when at most $\#$ of the $L^{p_j}$ spaces are equal to $L^{\infty}$ and the rest
are either $\overline{\S^{\infty}}$ or correspond to an index $j$ for which $1<p_j<\infty$. The goal would be to prove that $S(d)$ holds true, assuming $S(0)$ (which we promised to prove later on). To show that
$S(\#)$ implies $S(\#+1)$, let us fix some indices $1<p_1, ..., p_{d+1}\leq \infty$ as in the hypothesis of Theorem \ref{th1}. Since the argument is completely symmetric (in particular all the adjoints can be treated similarly) we can assume without loss
of generality that we want to extend (\ref{6}) for $C_d$, when the first $\#+1$ functions $f_1, ..., f_{\#+1}$ belong to $L^{\infty}$ while the other $\phi_{\#+2}, ..., \phi_{d+1}$ are Schwartz functions.

\subsection*{Case 1: $p>1$}

Since in this case $(L^p)^{\ast} = L^{p'}$ for $1/p+1/p' = 1$, one can simply use duality and define $C_d(f_1, ..., f_{\#+1}, \phi_{\#+2}, ..., \phi_{d+1})$ to be the unique $L^p$ function with the property that

$$\int_{\R}C_d(f_1, ..., f_{\#+1}, \phi_{\#+2}, ..., \phi_{d+1})(x) \phi_{d+2}(x) d x = 
\int_{\R}C_d^{\ast 1}(f_2, ..., f_{\#+1}, \phi_{\#+2}, ..., \phi_{d+1}, \phi_{d+2})(x) f_1(x) d x
$$
for any Schwartz function $\phi_{d+2}$ with $\|\phi_{d+2}\|_{p'} = 1$. This is clearly well defined as a consequence of $S(\#)$ for $C_d^{\ast 1}$.

\subsection*{Case 2: $p=1$}

In this case one has to be a bit careful since this time the dual of $L^1$ is $L^{\infty}$ and the Schwartz functions are no longer dense in it. However, we first observe that since $p=1$, there must be at least two indices $j_1$ and $j_2$ for which 
$1<p_{j_1}, p_{j_2} <\infty$. Again, by the symmetry of the argument, assume that these indices are precisely $\#+2$ and $\#+3$. To define properly $C_d(f_1, ..., f_{\#+1}, \phi_{\#+2}, ..., \phi_{d+1})$ as an element of
$L^1$, we first observe that one can do this as an element of (say) $L^2$, by taking advantage of the fact that all the functions $\phi_j$ are Schwartz and therefore belong to all the $L^q$ spaces simultaneously, for $1<q<\infty$.
Indeed, one can for instance think of
$\phi_{\#+2}, \phi_{\#+3}$ as being in $L^4$ while the rest of $\Phi_j$ all lie in $\overline{\S^{\infty}}$ and then define $C_d(f_1, ..., f_{\#+1}, \phi_{\#+2}, ..., \phi_{d+1})$ as being the unique function in $L^2$ with the property that

$$\int_{\R}C_d(f_1, ..., f_{\#+1}, \phi_{\#+2}, ..., \phi_{d+1})(x) \phi_{d+2}(x) d x = 
\int_{\R}C_d^{\ast 1}(f_2, ..., f_{\#+1}, \phi_{\#+2}, ..., \phi_{d+1}, \phi_{d+2})(x) f_1(x) d x
$$
exactly as before, for any $\phi_{d+2}$ Schwartz function with $\|\phi_{d+2}\|_2 = 1$, since one can rely again on $S(\#)$ for $C_d^{\ast 1}$.

Now that we know that $C_d(f_1, ..., f_{\#+1}, \phi_{\#+2}, ..., \phi_{d+1})$ is a well defined $L^2$ function, we would like to prove that it is in fact in $L^1$, as desired. One can write, for any big $M > 0$

\begin{equation}\label{9}
\int_{-M}^M \left|C_d(f_1, ..., f_{\#+1}, \phi_{\#+2}, ..., \phi_{d+1})(x)\right| d x = \int_{\R}C_d(f_1, ..., f_{\#+1}, \phi_{\#+2}, ..., \phi_{d+1})(x) \widetilde{\chi}_{[-M, M]}(x) d x
\end{equation}
where $|\widetilde{\chi}_{[-M, M]}(x)| = \chi_{[-M, M]}(x)$ almost everywhere.

Pick now a smooth and compactly supported sequence $(f_{d+2}^n)_n$ so that $f_{d+2}^n \rightarrow \widetilde{\chi}_{[-M, M]}(x)$ weakly and so that $\|f_{d+2}^n\|_{\infty} \leq 1$ (one can simply convolve $\widetilde{\chi}_{[-M, M]}$ with a smooth
approximation of identity, to obtain such a sequence). In particular, one can then majorize (\ref{9}) by

$$\lim_n \left|\int_{\R} C_d(f_1, ..., f_{\#+1}, \phi_{\#+2}, ..., \phi_{d+1})(x)  f_{d+2}^n (x) d x \right|\leq$$

$$
\sup_n \left|\int_{\R} C_d^{\ast 1} (f_2, ..., f_{\#+1}, \phi_{\#+2}, ..., \phi_{d+1}, f_{d+2}^n)(x) f_1(x) d x \right|
$$
and since now $f_{d+2}^n \in \overline{\S^{\infty}}$ and $\|f_{d+2}^n\|_{\infty} \leq 1$, one can again use the induction hypothesis to complete the argument.

Also, a careful look at the whole duality procedure, shows that if we assume (\ref{6}) with $C(d)$ growing polynomially, then this will be preserved after replacing all the $\overline{\S^{\infty}}$ by the corresponding $L^{\infty}$.

We are thus left with proving (\ref{6}) for $C_d$ and its adjoints. The advantage of it is that when applied to Schwartz functions, all the operators $C_d^{\ast i}$ for $1\leq i\leq d+2$,
are given by well defined expressions similar to (\ref{4}). Later on, they will be furher decomposed and discretized carefully, and this will allow us to reduce (\ref{6}) even more, 
to some similar estimates, but for finite and well localized {\it model operators}.

\section{Logarithmic estimates and discrete models}

In this section the goal is to describe some logarithmic estimates for certain very concrete discrete model operators, which will play an important role in proving the desired (\ref{6}).

In order to motivate them, and also to get a general idea of the strategy of the proof, let us assume for simplicity that instead of (\ref{6}), one would like to prove 
$L^p\times L^{\infty}\times ... \times L^{\infty} \rightarrow L^p$ estimates (say) for a generic paraproduct $\Pi_{d+1}(f_1, ... , f_{d+1})$ whose $(d+2)$- linear form is given by

\begin{equation}\label{11}
\int_{\R} \sum_k (f_1\ast \Phi^1_k)(x) ... (f_{d+2}\ast \Phi^{d+2}_k)(x) d x
\end{equation}
where $f_1\in L^p$, $f_j\in L^{\infty}$ for $2\leq j\leq d+1$, while $f_{d+2}\in L^{p'}$ with $1/p+1/p' = 1$.

As  usual, all the functions $(\Phi^j_k)_k$ are smooth $L^1$ normalized bumps, adapted to intervals of the form $[-2^{-k}, 2^{-k}]$ for $k\in\Z$, and for at least two indices $j_1, j_2$ one has
$\int_{\R}\Phi^{j_1}_k(x) d x = \int_{\R}\Phi^{j_2}_k(x) d x    = 0$ (in which case we say that their corresponding families are of {\it $\Psi$ type}, while the others are of {\it $\Phi$ type}).

We witness several situations.

\subsection*{Case A: $j_1=1$ and $j_2 = d+2$}

Let us also assume that the $L^1$ norms of the functions in the {\it $\Phi$ families} are not only uniformly bounded, but they are bounded by $1$. In particular, for any $2\leq j\leq d+1$
(in this case) one has

\begin{equation}\label{unu}
|f_j\ast \Phi^j_k (x) | \leq \|\Phi^j_k\|_1 \cdot \|f_j\|_{\infty} \leq \|f_j\|_{\infty}.
\end{equation}
One can then majorize (\ref{11}) by

$$\prod_{j=2}^{d+1} \|f_j\|_{\infty}\cdot \int_{\R}\sum_k |f_1\ast \Phi^1_k (x)| |f_{d+2}\ast \Phi^{d+2}_k (x)| d x \leq $$

$$\prod_{j=2}^{d+1} \|f_j\|_{\infty} \cdot\int_{\R}
\left(\sum_k |f_1\ast \Phi^1_k (x)|^2\right)^{1/2}\cdot \left(\sum_k |f_{d+2}\ast \Phi^{d+2}_k (x)|^2\right)^{1/2} d x =$$

$$\prod_{j=2}^{d+1} \|f_j\|_{\infty} \cdot \int_{\R} S(f_1)(x) S(f_{d+2})(x) d x \leq$$

$$\prod_{j=2}^{d+1} \|f_j\|_{\infty}\cdot      
\|S(f_1)\|_p \cdot \|S(f_{d+2})\|_{p'} \lesssim \prod_{j=2}^{d+1} \|f_j\|_{\infty} \cdot \|f_1\|_p \cdot \|f_{d+2}\|_{p'}
$$
as desired, by using the fact that the Littlewood-Paley square function $S$ is a bounded operator on any $L^q$ space, for $1< q < \infty$.

\subsection*{Case B: $j_1 = 1$ and $j_2 = 2$}

This case will not be that simple. This time, one can only majorize (\ref{11}) by

$$\prod_{j=3}^{d+1} \|f_j\|_{\infty}\cdot \int_{\R}\sum_k |f_1\ast \Phi^1_k (x)| |f_2\ast \Phi^2_k (x)| |f_{d+2}\ast \Phi^{d+2}_k (x)| d x \leq $$

$$\prod_{j=3}^{d+1} \|f_j\|_{\infty} \cdot\int_{\R} \left(\sum_k |f_1\ast \Phi^1_k (x)|^2\right)^{1/2}\cdot \left(\sum_k |f_2\ast \Phi^2_k (x)|^2\right)^{1/2} \cdot \left( \sup_k |f_{d+2}\ast \Phi^{d+2}_k (x)|\right)   d x =$$

$$\prod_{j=3}^{d+1} \|f_j\|_{\infty} \cdot \int_{\R} S(f_1)(x) S(f_2)(x) M(f_{d+2})(x)  d x \leq$$

$$\prod_{j=3}^{d+1} \|f_j\|_{\infty} \cdot \|S(f_1)\|_{s_1}\cdot \|S(f_2)\|_{s_2}\cdot \|M(f_{d+2})\|_{s'_3}\lesssim $$

$$\prod_{j=3}^{d+1} \|f_j\|_{\infty} \cdot \|f_1\|_{s_1}\cdot \|f_2\|_{s_2}\cdot \|f_{d+2}\|_{s'_3},$$
for any $1< s_1, s_2, s_3 <\infty$ so that $1/s_1 + 1/s_2 = 1/s_3$, 
by using the fact that besides $S$, the Hardy-Littlewood maximal function $M$ is also bounded on any $L^q$ space, for $1< q <\infty$.
Clearly, the estimate we are looking for corresponds to $s_1 = s_3 = p$ and $s_2 = \infty$ and it cannot be obtained in this way, since $S$ is unbounded on $L^{\infty}$.

If on the other hand one freezes the functions $f_3, ..., f_{d+1}$, expression (\ref{11}) becomes a $3$- linear form and the above estimates show that its associated bilinear operator
$\Pi_{2}(f_1, f_2)$ is bounded from $L^{s_1}\times L^{s_2}$ into $L^{s_3}$. By symmetry, the same is true for both $\Pi_2^{\ast 1}$ and $\Pi_2^{\ast 2}$. The estimate we are interested in,
can then be rephrased as 

\begin{equation}\label{00}
\Pi_2 : L^p \times L^{\infty} \rightarrow L^p.
\end{equation}
To get it, one needs besides the previous Banach estimates to prove quasi-Banach estimates as well, of the form $\Pi_2^{\ast 2}: L^{r_1}\times L^{r_2} \rightarrow L^{r_3}$, for any
$1< r_1, r_2 <\infty$, $0< r_3 < \infty$ with $1/r_1 + 1/r_2 = 1/r_3$. In the case of paraproducts, there are several ways to achieve that, see for instance \cite{meyerc}. In the end, one can 
use multi-linear interpolation between the Banach and quasi-Banach estimates as in \cite{mtt:general}, to obtain the intermediate (\ref{00}). Even more precisely, the convexity argument in \cite{mtt:general} shows that
there exist two Banach and one quasi-Banach estimates (with implicit boundedness constants $C_B^1$, $C_B^2$, $C_{q-B}$ respectively) so that if one denotes by 
$C_{B} := \max \{C_{B}^1, C_{B}^2, C_{q-B} \}$, one has that this constant represents an upper bound for the boundedness constant of (\ref{00}).
 
\subsection*{Case C: $j_1 = 2$ and $j_2 = 3$}

Finally, we are left with this situation which can be treated similarly to the previous one. More precisely, one can majorize (\ref{11}) this time by

$$\prod_{j=4}^{d+1} \|f_j\|_{\infty}\cdot \int_{\R}\sum_k |f_1\ast \Phi^1_k (x)| |f_2\ast \Phi^2_k (x)| |f_3\ast \Phi^3_k (x)| |f_{d+2}\ast \Phi^{d+2}_k (x)| d x \leq $$

$$\prod_{j=4}^{d+1} \|f_j\|_{\infty} \cdot\int_{\R}
\left(\sum_k |f_2\ast \Phi^2_k (x)|^2\right)^{1/2} \left(\sum_k |f_3\ast \Phi^3_k (x)|^2\right)^{1/2}
\left( \sup_k |f_1\ast \Phi^1_k (x)|\right)
\left( \sup_k |f_{d+2}\ast \Phi^{d+2}_k (x)|\right)   d x =$$

$$\prod_{j=4}^{d+1} \|f_j\|_{\infty} \cdot \int_{\R} S(f_2)(x) S(f_3)(x) M(f_1)(x) M(f_{d+2})(x)  d x \leq$$

$$\prod_{j=4}^{d+1} \|f_j\|_{\infty} \cdot \|M(f_1)\|_{s_1}\cdot \|S(f_2)\|_{s_2}\cdot  \|S(f_3)\|_{s_3}\cdot \|M(f_{d+2})\|_{s'_4}\lesssim $$

$$\prod_{j=4}^{d+1} \|f_j\|_{\infty} \cdot \|f_1\|_{s_1}\cdot \|f_2\|_{s_2}\cdot   \|f_3\|_{s_3}    \|f_{d+2}\|_{s'_4},$$
for any $1< s_1, s_2, s_3 , s_4 <\infty$ so that $1/s_1 + 1/s_2 + 1/s_3= 1/s_4$.

The estimate we are looking for corresponds to $s_1 = s_4 = p$ and $s_2 = s_3 = \infty$ and as before, it cannot be obtained in this way.

This time one freezes the functions $f_4, ..., f_{d+1}$ and then expression (\ref{11}) becomes a $4$- linear form and the estimates above show that its associated $3$- linear operator
$\Pi_{3}(f_1, f_2, f_3)$ is bounded from $L^{s_1}\times L^{s_2} \times L^{s_3}$ into $L^{s_4}$. By symmetry, the same is true for its adjoints $\Pi_3^{\ast 1}$, $\Pi_3^{\ast 2}$ and $\Pi_3^{\ast 3}$.  The estimate we are interested in,
becomes 

\begin{equation}\label{000}
\Pi_3 : L^p \times L^{\infty} \times L^{\infty} \rightarrow L^p.
\end{equation}
As in the previous case, to get it, one needs besides the previous Banach estimates to prove quasi-Banach estimates as well, of the form $\Pi_3^{\ast 2}, \Pi_3^{\ast 3}: L^{r_1}\times L^{r_2} \times L^{r_3} \rightarrow L^{r_4}$, for any
$1< r_1, r_2 , r_3<\infty$, $0< r_4 < \infty$ with $1/r_1 + 1/r_2 + 1/r_3 = 1/r_4$.  And in the end, one can again
use multi-linear interpolation between the Banach and quasi-Banach estimates, to obtain the intermediate (\ref{000}). In particular, the same convexity argument shows that
there exist two Banach and two quasi-Banach estimates (with implicit boundedness constants $C_B^1$, $C_B^2$, $C_{q-B}^1$, $C_{q-B}^2$   respectively) so that if one denotes by 
 $C_{B} := \max \{C_{B}^1, C_{B}^2, C_{q-B}^1, C_{q-B}^2\}$, one has that this constant
is an upper bound for the boundedness constant of (\ref{000}). And this ends the discussion on the boundedness of $\Pi_{d+1}$ from $L^p\times L^{\infty}\times ... \times L^{\infty}$ into $L^p$ since 
by symmetry, it is easy to realize that any other possibility can be reduced to
one of these cases \footnote{It should also be clear that a similar argument works in the general $\Pi_{d+1}: L^{p_1}\times ... \times L^{p_{d+1}} \rightarrow L^p$ case. Instead of the {\it minimal}
bilinear or trilinear operators which appeared before, one would have to deal with $l$- linear ones for some $1\leq l\leq d+1$, but the interpolation between the natural corresponding 
Banach and quasi-Banach estimates, works precisely in the same way. }.
$$$$
There are a couple of important facts that one learns from
the previous argument. First, the bounds are independent of $d$. Responsible for this is the assumption that the $L^1$ norms of the {\it $\Phi$ families} are all at most $1$, which implied
the crucial (\ref{unu}). Then, there is the fact that after using (\ref{unu}) several times, we reduced our analysis to the study of several (Banach or quasi-Banach) corresponding estimates, 
for some  {\it minimal} bilinear or tri-linear operators. 

We claim now that in spite of the fact that $C_d$ is not a Coifman-Meyer operator, it can be studied in an analogous manner. More precisely, one can decompose it first into
polynomially (in $d$) many {\it paraproduct like pieces}, and then estimate each piece independently of $d$. And also as before (since the Banach estimates are easy) we will reduce
the main inequality (via interpolation), to similar quasi-Banach estimates for {\it minimal} $l$- linear operators, for some $1\leq l\leq d+1$. The proof of the precise quasi-Banach estimates
is in general a delicate issue, but it has already been discussed in detail in \cite{camil}.

The necessity to discretize the {\it minimal} $(l+1)$- linear forms justifies the introduction of the following model operators.

Fix then $l$ a positive integer, $\overline{n_1}, ..., \overline{n_l}$ arbitrary integers and consider families 
$(\Phi^1_{I_{\overline{n_1}}})_I$, $(\Phi^2_{I_{\overline{n_2}}})_I$, ... ,$(\Phi^{l}_{I_{\overline{n_{l}}}})_I$ of $L^2$ normalized bumps 
adapted to dyadic intervals
$I_{\overline{n_j}}$ (as in \cite{camil}, given $I$, denote by $I_{\overline{n_j}}$ the interval of the same length as $I$, but sitting $\overline{n_j}$ units of length $|I|$ away from $I$)
so that at least two of them are of {\it $\Psi$ type} (i.e. their integrals are zero) 
By definition, a smooth function $\Phi$ is said to be adapted to an interval $I$, if one has 

$$|\partial^{\alpha}\Phi (x)| \lesssim \frac{1}{|I|^{|\alpha|}}\frac{1}{\left(1 + \frac{\dist (x, I)}{|I|}\right)^{M}}$$ for any derivative $\alpha$ so that
$|\alpha|\leq 5$ and any large $M>0$, with the implicit constants depending on it. Then, also by definition, if $1\leq p\leq \infty$, we say 
that $|I|^{-1/p} \Phi$ is $L^p$ normalized. 

Define the $l$ linear discrete operator $T_{\J}$ for $\J$ a finite family of dyadic intervals, by

\begin{equation}\label{12}
T_{\J}(f_1, ..., f_l) = \sum_{I\in\J}
\frac{1}{|I|^{(l-2)/2}}
\langle f_1, \Phi^1_{I_{\overline{n_1}}}\rangle ... 
\langle f_l, \Phi^l_{I_{\overline{n_l}}}\rangle 
\Phi^{l+1}_I.
\end{equation}
One has

\begin{theorem}\label {th2}
For any such a family $\J$, the $l$-linear operator $T_{\J}$ maps $L^{p_1}\times ... \times L^{p_l} \rightarrow L^p$ for any $1< p_1, ..., p_l < \infty$ with $1/p_1 + ... +1/p_l = 1/p$, $0<p<\infty$, 
with a bound of type

$$O(\log <\overline{n_1}> \cdot ... \cdot \log <\overline{n_l}>).$$
Here, as in \cite{camil}, $<\overline{n_j}>$ simply denotes $2+|\overline{n_j}|$. Also, the implicit constants above are allowed to depend on $l$.
\end{theorem}
This theorem is the $l$- linear generalization of the bilinear Theorem 2 in \cite{camil} and since its proof is identical to the proof of that theorem, we are leaving it to the reader.

More precisely (as in \cite{camil}), Theorem \ref{th2} follows (by scale invariance and interpolation) from the more precise fact 
that for every $f_j \in L^{p_j}$ with $\|f_j\|_{p_j} = 1$,  $1\leq j\leq l$
and measurable set $E \subseteq \R$ of measure $1$, there exists a subset $E'\subseteq E$ of comparable measure so that

\begin{equation}\label{q-b}
\sum_{J\in\J}
\frac{1}{|I|^{(l-1)/2}}
|\langle f_1, \Phi^1_{I_{\overline{n_1}}}\rangle| ...
|\langle f_l, \Phi^l_{I_{\overline{n_l}}}\rangle|
|\langle f_{l+1}, \Phi^{l+1}_I\rangle| \lesssim \log <\overline{n_1}> \cdot ... \cdot \log <\overline{n_l}> ,
\end{equation}
where $f_{l+1} = \chi_{E'}$.

As in \cite{camil}, the fact that one looses only logarithmic bounds in the estimates above, will be important later on.

In the rest of the paper we will describe the calculations that are necessary to show how the desired (\ref{6}) can be indeed reduced to (\ref{q-b}).

\section{Reduction to the discrete model}

We treat the case of $C_d$ only, since by the symmetry of the argument, all its adjoints can be understood in a similar way. Fix then indices $p_j$ for $1\leq j\leq d+1$ and $p$ as in the hypothesis of
Theorem \ref{th1}. As suggested before, the first step is to decompose $C_d$ into polynomially (in $d$) many {\it paraproduct like pieces} which will be analized afterwords.

Here, the clasical Littlewood-Paley decompositions will be of great help. However, since we want to have the {\it perfect inequalities} (\ref{unu}) available, we need to work most of the time with {\it non-compact} (in frequency) approximations of identity, 
which will cause several technical dificulties later on. We define them in detail in the next subsection.

\subsection*{Non-compact Littlewood-Paley $L^1$ normalized projections}

Start with a Schwartz function $\Phi(x)$ which is even, positive and so that $\int_{\R} \Phi(x) d x = 1$. Define also $\Psi(x)$ by

$$\Psi(x) = \Phi(x) - 1/2 \Phi(x/2)$$
and observe that $\int_{\R}\Psi(x) d x  = 0$.
Then, as usual, $\Psi_k(x)$ and $\Phi_k(x)$ denote $2^k\Psi(2^k x)$ and $2^k \Phi(2^k x)$ respectively. Notice that all the $L^1$ norms of $\Phi_k$ are equal to $1$. Observe also that one has

$$\Psi_k(x) = \Phi_k(x) - \Phi_{k-1}(x),$$
and then it is easy to see that

$$\sum_{k\leq \overline{k}}\Psi_k = \Phi_{\overline{k}}$$
and also that

\begin{equation}
\sum_{k\in\Z}\Psi_k = \delta_0
\end{equation}
or equivalently

\begin{equation}\label{13}
\sum_{k\in\Z}\widehat{\Psi_k}(\xi) = 1
\end{equation}
for almost every $\xi\in\R$. On the other hand, 

$$\widehat{\Psi}(0) = \int_{\R}\Psi(x) d x = 1 - 1 = 0.$$
Moreover, one also has that

$$\widehat{\Psi}'(0) = - 2\pi i \int_{\R} x \Psi(x) d x = 0$$
by using the fact that $\Phi$ is an even function. As a consequence, one can write $\widehat{\Psi}(\xi)$ as 

\begin{equation}\label{psi}
\widehat{\Psi}(\xi) = \xi^2 \widehat{\phi}(\xi)
\end{equation}
for another smooth and rapidly decaying function $\phi$.

These are the {\it non-compact} $L^1$ normalized Littlewod-Paley decompositions. The {\it compact} ones are obtained in a similar manner, but one starts instead with a Schwartz function $\Phi$ having the property 
that $\supp \,\widehat{\Phi} \subseteq [-1, 1]$ and $\widehat{\Phi}(\xi) = 1$ on the subinterval $[-1/2, 1/2]$.

\subsection*{Some remarks on the symbols of $C_d$ for $d\geq 2$}

Before going any further, it is worthwhile to have a look at the symbol of the second commutator $C_2$. It is very natural to try to see if its Fourier coefficients satisfy the same 
{\it quadratic estimates} (proved in \cite{camil}) as the symbol of $C_1$. Consider for instance three Schwartz functions $\widehat{\phi}(\xi)$, $\widehat{\phi}(\xi_1)$ and $\widehat{\phi}(\xi_2)$
supported inside the intervals $[-2,-1]$, $[1,2]$ and $[-1/2, 1/2]$ respectively. Clearly, the function

$$(\xi, \xi_1, \xi_2) \rightarrow \widehat{\phi}(\xi) \widehat{\phi}(\xi_1) \widehat{\phi}(\xi_2)$$
is supported inside a Whitney cube (with respect to the origin) in $\R^3$ and the goal is to understand the expression

\begin{equation}\label{f-coef}
\int_{\R^3}
\left[\int_{[0,1]^2} 1_{\R_{+}}(\xi +\alpha\xi_1 +\beta \xi_2) d\alpha d\beta\right]
\widehat{\varphi}(\xi)
\widehat{\varphi}(\xi_1)
\widehat{\phi}(\xi_2)\cdot
\end{equation}

$$
\cdot e^{-2\pi i n\xi}
e^{-2\pi i n_1\xi_1}
e^{-2\pi i n_2\xi_2} d\xi d\xi_1 d\xi_2
$$
when $n, n_1, n_2$ are arbitrary integers. Since $\xi_1$ can never be zero, the symbol in \eqref{f-coef} can be rewritten as

\begin{equation}\label{new-form}
\int_0^1 \frac{1}{\xi_1} \int_0^{\xi_1} 1_{\R_{+}}(\xi +\alpha +\beta \xi_2) d\alpha d\beta.
\end{equation}
When one differentiates \eqref{new-form} with respect to $\xi_1$ the inner term becomes

$$\int_0^1 1_{\R_{+}} (\xi+\xi_1 + \beta \xi_2) d \beta$$
which coincides to $m_1(\xi+\xi_1, \xi_2)$. The important difference now is that since $\xi+\xi_1$ lies inside the interval $[-1,1]$ and $\xi_2$ inside $[-1/2, 1/2]$ and they both contain the origin, one can no longer continue to apply the argument in \cite{camil}.

As a consequence of this particular example, the Fourier coefficients in \eqref{f-coef} seem to decay only {\it linearly}, which is clearly not enough. In any event, these comments show that passing from the
analysis of the first commutator to the analysis of the second one and all the rest, is not at all a {\it routine} task. One should recall that there are ten years difference between the results of Calder\'{o}n \cite{calderon} and the ones of
Coifman and Meyer \cite{meyerc1}.

On the other hand, this also shows that from this point of view at least, the symbol of $C_2$ looks similar to the symbol of the bilinear Hilbert transform (given by $\sgn(\xi_1-\xi_2))$ whose Fourier coefficients decay also only {\it linearly}
as one can easily check. This fact might also be considered as another {\it possible explanation} of why Calder\'{o}n suggested the study of the bilinear Hilbert transform as a step towards understanding all his commutators, besides
the one recalled already in \cite{camil}.

\subsection*{The generic decomposition of $C_d$}

Coming back to our goal, notice first that because of (\ref{4}) if $f, f_1, ..., f_{d+1}$ are all Schwartz functions, one can write the $(d+2)$- linear form $\Lambda_d (f, f_1, ..., f_{d+1})$ associated with $C_d$ as

\begin{equation}\label{form}
\int_{\xi+\xi_1+ ... +\xi_{d+1} = 0}
\left[\int_{[0,1]^d} 1_{\R^{+}}(\xi+\alpha_1\xi_1 + ... +\alpha_d \xi_d) d\alpha_1 ... d\alpha_d \right]
\widehat{f}(\xi)
\widehat{f_1}(\xi_1) ...
\widehat{f_{d+1}}(\xi_{d+1}) d\xi d\xi_1 ... d\xi_{d+1}.
\end{equation}
By combining several Littlewood-Paley decompositions as in (\ref{13}), one can write

\begin{equation}\label{14}
1 = \sum_{k_0, k_1, ..., k_d, k_{d+1} \in \Z}
\widehat{\Psi_{k_0}}(\xi)
\widehat{\Psi_{k_1}}(\xi_1) ... 
\widehat{\Psi_{k_d}}(\xi_d)
\widehat{\Psi_{k_{d+1}}}(\xi_{d+1}).
\end{equation}
Now, for every $(d+2)$- tuple $(k_0, k_1, ..., k_d, k_{d+1}) \in \Z^{d+2}$, one has either $k_0 \geq k_1, ..., k_d, k_{d+1}$ or $k_1\geq k_0, k_2, ..., k_{d+1}$ ... or $k_{d+1}\geq k_0, k_1, ..., k_d$. By replacing some of these inequalities
with {\it strict inequalities}, one can also make sure that the corresponding $d+2$ {\it regions} in $\Z^{d+2}$ are all disjoint.

Fixing always the biggest
parameter and summing over the rest of them, one can rewrite the constant $1$ in (\ref{14}) as\footnote{To be totally rigorous, some of these $\Phi_k$ functions should be $\Phi_{k-1}$, but this is a minor issue.}

$$\sum_k \widehat{\Psi_k}(\xi) \widehat{\Phi_k}(\xi_1) ... \widehat{\Phi_k}(\xi_d) \widehat{\Phi_k}(\xi_{d+1}) +$$

$$ ... $$

\begin{equation}\label{15}
\sum_k \widehat{\Phi_k}(\xi) \widehat{\Phi_k}(\xi_1) ... \widehat{\Phi_k}(\xi_d) \widehat{\Psi_k}(\xi_{d+1}).
\end{equation}
There are $d+2$ {\it inner terms} in the decomposition (\ref{15}) each containing a single $\Psi$ type of a function. 

For some technical reasons that will be clearer later on, we assume that for the $\xi$ and $\xi_{d+1}$ variables we use 
{\it compact} Littlewood-Paley decompositions, while for the rest the {\it non-compact} one in (\ref{13}).

Let us assume now that in addition, one has $\xi +\xi_1 + ... +\xi_{d+1} = 0$. Look at the second (for instance) sum in (\ref{15}) and consider the $k=0$ inner term  which we write for simplicity as

\begin{equation}\label{16}
\widehat{\Phi}(\xi) \widehat{\Psi}(\xi_1) ... \widehat{\Phi}(\xi_d) \widehat{\Phi}(\xi_{d+1}).
\end{equation}
Since from (\ref{psi}) we know that $\widehat{\Psi}(\xi_1) = \xi_1^2 \widehat{\phi}(\xi_1)$, one can rewrite this as

\begin{equation}\label{17}
\widehat{\Psi}(\xi_1) = \xi_1 \widehat{\phi}(\xi_1) (-\xi - \xi_2 - ... -\xi_{d+1}) =
-\xi_1 \xi\widehat{\phi}(\xi_1) - \xi_1\xi_2 \widehat{\phi}(\xi_1) - ... - \xi_1\xi_{d+1}\widehat{\phi}(\xi_1). 
\end{equation}
Using this in (\ref{16}), one can write it as another sum of $O(d)$ terms, containing this time two functions of $\Psi$ type, since besides $\xi_1\widehat{\phi}(\xi_1)$, one obtains in addition either
expressions of type $\xi_j\widehat{\Phi}(\xi_j)$ for $j=2, ..., d+1$, or $\xi \widehat{\Phi}(\xi)$.

If one does this for every scale $k\in\Z$ and every inner term in (\ref{15}), one obtains a decomposition of $1_{\{\xi +\xi_1 + ... +\xi_{d+1} = 0\}}$ as a sum of $O(d^2)$ expressions whose generic inner product terms all
contain two functions of $\Psi$ type, which are more specifically of the form $\gamma\widehat{\phi}(\gamma)$\footnote{And as a consequence, the Schwartz functions whose Fourier transforms are given by such  
expression, have integral zero.}. 
If one inserts this into the formula for the $(d+2)$- linear form (\ref{form}), one obtains
$O(d^2)$ $(d+2)$- linear forms which will be carefully analized next. This is our {\it generic decomposition}. To be able to go further, one needs to understand how to {\it unfold} the symbol of $C_d$.

As before in the case of paraproducts, the positions of the $\Psi$ functions (we denote them by $j_1, j_2$ for $0\leq j_1, j_2 \leq d+1$) will play an important role.  
There are in fact three distinct cases, depending on where these $\Psi$ functions lie.

\subsection*{Case 1: $j_1 = 0$ and $j_2 = 1$}

For symmetry, we change notations and rewrite the $(d+2)$- linear form as

\begin{equation}\label{18}
\sum_k \int_{\xi+\xi_1+ ... +\xi_{d+1} = 0}
[\int_{[0,1]^d} 1_{\R^{+}}(\xi+\alpha_1\xi_1 + ... +\alpha_d \xi_d) d\alpha_1 ... d\alpha_d ]\cdot
\end{equation}

$$
\widehat{\Phi^0_k}(\xi)
\widehat{\Phi^1_k}(\xi_1) ... 
\widehat{\Phi^d_k}(\xi_d)
\widehat{\Phi^{d+1}_k}(\xi_{d+1})\cdot
\widehat{f}(\xi)
\widehat{f_1}(\xi_1) ...
\widehat{f_{d+1}}(\xi_{d+1}) d\xi d\xi_1 ... d\xi_{d+1}
$$
with both families $(\Phi^0_k)_k$ and $(\Phi^1_k)_k$ being of $\Psi$ type. Without the $C_d$ symbol above, the expression (\ref{18}) would be the $(d+2)$- linear form of a paraproduct, which could be analyzed
as we described earlier. The first impulse to deal with it, is to try to decompose it into multiple Fourier series, on the support of the corresponding Whitney {\it frequency boxes}. However,
it is clear that this will produce in the end an upper bound of the type of a product of $O(d)$ power series, which will grow exponentially in $d$ even in the case of classical symbols, so one has to be very careful at this stage. 
The situation is in fact even worse as we pointed out a bit earlier,
since it seems that the Fourier coefficients of the symbol of $C_d$ for $d\geq 2$, do not decay quadratically as the ones of the symbol of $C_1$.

The idea now is to realize that the variable $\xi_1$ (in this case) is in some sense {\it special} and the right thing to do is to look at the $C_d$ symbol as being a multiple average of $C_1$ symbols
(depending on $\xi_1$ and on a new variable $\widetilde{\xi}$), which can be analysed as in \cite{camil}. 

To be able to {\it execute} this plan, since the functions $\widehat{\Phi^j_k}(\xi_j)$ for $1\leq j\leq d$ do not have compact support, one has to insert yet two other compact Littlewood-Paley 
decompositions 
of the unity in (\ref{18}). More precisely, denote
by $\widetilde{\xi} = \xi +\alpha_2\xi_2 + ... +\alpha_d \xi_d$ and write

\begin{equation}\label{extraLP}
1 = \sum_{k_0, k_1} \widehat{\Psi_{k_0}}(\widetilde{\xi}) \widehat{\Psi_{k_1}}(\xi_1) = \sum_{k_0 << k_1} ... + \sum_{k_0 \sim \k_1} ... + \sum_{k_0 >> k_1} ...
\end{equation}
which can be rewritten as before as

\begin{equation}\label{19}
\sum_r \widehat{\Phi_r}(\widetilde{\xi})\widehat{\Psi_r}(\xi_1) + \sum_r \widehat{\Psi_r}(\widetilde{\xi})\widehat{\Psi_r}(\xi_1) + \sum_r \widehat{\Psi_r}(\widetilde{\xi})\widehat{\Phi_r}(\xi_1).
\end{equation}
To be totally rigorous, one should have had $\widehat{\Phi_{r-100}}$ instead of $\widehat{\Phi_r}$ in (\ref{19}) and also finitely many {\it middle $\Psi$ - $\Psi$ terms}, instead of only one. We should keep that
in mind but leave it as that, for the simplicity of our notation.

As a consequence, if we insert (\ref{19}) into (\ref{18}), it splits as a sum of three distinct expressions which will be analized separately. We will denote this cases by $1_a$, $1_b$ and $1_c$ respectively.

\subsection*{Case $1_a$}

To see the effect of the new splitting (over $r$) in (\ref{18}), let's analize for simplicity the particular term corresponding to $k=0$. If we ignore the symbol $\int_{[0,1]^d} 1_{\R^{+}}(\xi+\alpha_1\xi_1 + ... +\alpha_d \xi_d) d\alpha_1 ... d\alpha_d $
for a while, the rest of the expression becomes

$$\sum_r \left[\widehat{\Phi_r}(\widetilde{\xi})\widehat{\Psi_r}(\xi_1)\right]
\widehat{\Phi^0_0}(\xi)
\widehat{\Phi^1_0}(\xi_1) ...
\widehat{\Phi^d_0}(\xi_d)
\widehat{\Phi^{d+1}_0}(\xi_{d+1}) =
$$

\begin{equation}\label{20}
\sum_{r\leq  0} ... + \sum_{r > 0} ...  = 1'_a + 1_a''.
\end{equation}

\subsection*{Case $1_a'$}

Using the fact that $\widehat{\Psi_r}(\xi_1)$ is compactly supported and taking also into account the fact that $\widehat{\Phi^1_0}(\xi_1)$ is also of $\Psi$ type (more precisely, as we have seen, it is of the form $\xi_1\widehat{\phi}(\xi_1)$), 
one can rewrite $1'_a$ as

$$\sum_{r\leq 0}
2^r
\widehat{\Phi_r}(\widetilde{\xi})
\widehat{\Phi^0_0}(\xi)
\widehat{\Psi^1_r}(\xi_1) ... 
\widehat{\Phi^d_0}(\xi_d)
\widehat{\Phi^{d+1}_0}(\xi_{d+1}) = 
$$

$$\sum_{r\leq 0}
2^r
\left[\widehat{\widetilde{\Phi_r}}(\widetilde{\xi})\widehat{\widetilde{\Psi^1_r}}(\xi_1) \right]
\widehat{\Phi^0_0}(\xi)
\widehat{\Psi^1_r}(\xi_1) ... 
\widehat{\Phi^d_0}(\xi_d)
\widehat{\Phi^{d+1}_0}(\xi_{d+1})\widehat{\Phi_r}(\widetilde{\xi})
$$
for certain compactly supported well chosen functions $\widehat{\widetilde{\Phi_r}}(\widetilde{\xi})$, $\widehat{\Psi^1_r}(\xi_1)$ and $\widehat{\widetilde{\Psi^1_r}}(\xi_1)$ (naturally, the first is of $\Phi$ type, while the other two are of $\Psi$ type).

In particular, we can split the symbol 

$$\left[\int_0^1 1_{\R^{+}}(\widetilde{\xi} + \alpha_1\xi_1) d\alpha_1 \right]\widehat{\widetilde{\Phi_r}}(\widetilde{\xi})\widehat{\widetilde{\Psi^1_r}}(\xi_1)$$
as a double Fourier series of the form

\begin{equation}\label{21}
\sum_{n, n_1} C^r_{n, n_1} e^{2 \pi i\frac{n}{2^r}\widetilde{\xi}} e^{2 \pi i \frac{n_1}{2^r}\xi_1},
\end{equation}
where

$$C^r_{n, n_1} =
\frac{1}{2^r}
\frac{1}{2^r}
\int_{\R^2}
\left[\int_0^1 1_{\R^{+}}(\widetilde{\xi} + \alpha_1\xi_1) d\alpha_1 \right]
\widehat{\widetilde{\Phi_r}}(\widetilde{\xi})\widehat{\widetilde{\Psi^1_r}}(\xi_1)
e^{- 2 \pi i\frac{n}{2^r}\widetilde{\xi}} e^{- 2 \pi i \frac{n_1}{2^r}\xi_1} d\widetilde{\xi} d \xi_1 =
$$

$$
\int_{\R^2}
\left[\int_0^1 1_{\R^{+}}(\widetilde{\xi} + \alpha_1\xi_1) d\alpha_1 \right]
\widehat{\widetilde{\Phi_0}}(\widetilde{\xi})\widehat{\widetilde{\Psi^1_0}}(\xi_1)
e^{- 2 \pi i n \widetilde{\xi}} e^{- 2 \pi i n_1 \xi_1} d\widetilde{\xi} d \xi_1
$$
which is an expression independent of $r$.

Recall now from \cite{camil} the crucial fact that

\begin{equation}\label{22}
|C^r_{n, n_1}|  ( = |C_{n, n_1}|) \lesssim \frac{1}{<n>^2} \frac{1}{<n_1>^{\#}}
\end{equation}
for any large number $\# > 0$.

These calculations show that the corresponding contribution of $1'_a$ in (\ref{18}) is

$$\int_0^1 ... \int_0^1
\sum_{r\leq 0}2^r
\sum_{n, n_1} C^r_{n, n_1}\cdot
$$

$$
\int_{\xi+\xi_1 + ... +\xi_d +\xi_{d+1} = 0}
\left[\widehat{\Phi^0_0}(\xi) e^{2 \pi i \frac{n}{2^r}\xi}\right]\cdot
\left[\widehat{\Psi^1_r}(\xi_1) e^{2 \pi i \frac{n_1}{2^r}\xi_1}\right]\cdot
\left[\widehat{\Phi^2_0}(\xi_2) e^{2 \pi i \frac{n\alpha_2}{2^r}\xi_2}\right]\cdot ... \cdot
\left[\widehat{\Phi^d_0}(\xi_d) e^{2 \pi i \frac{n\alpha_d}{2^r}\xi_d}\right]
$$

\begin{equation}\label{23}
\widehat{\Phi^{d+1}_0}(\xi_{d+1})\widehat{\Phi_r}(\widetilde{\xi})
\widehat{f}(\xi)\widehat{f_1}(\xi_1) ... \widehat{f_d}(\xi_d) \widehat{f_{d+1}}(\xi_{d+1}) d \xi d \xi_1 ... d \xi_{d+1} d \alpha_2 ... d\alpha_d.
\end{equation}
If one fixes now $\alpha_2, ..., \alpha_d \in [0,1]$,  $r$ and $n, n_1$, the inner expression becomes

$$
\int_{\xi+\xi_1 + ... +\xi_d +\xi_{d+1} = 0}
[\widehat{f}(\xi)\widehat{\Phi^0_0}(\xi) e^{2 \pi i \frac{n}{2^r}\xi}]\cdot
[ \widehat{f_1}(\xi_1) \widehat{\Psi^1_r}(\xi_1) e^{2 \pi i \frac{n_1}{2^r}\xi_1}]\cdot
$$

\begin{equation}\label{24}
[\widehat{f_2}(\xi_2) \widehat{\Phi^2_0}(\xi_2) e^{2 \pi i \frac{n\alpha_2}{2^r}\xi_2}]\cdot ...\cdot
[\widehat{f_d}(\xi_d) \widehat{\Phi^d_0}(\xi_d) e^{2 \pi i \frac{n\alpha_d}{2^r}\xi_d}]\cdot
[\widehat{f_{d+1}}(\xi_{d+1}) \widehat{\Phi^{d+1}_0}(\xi_{d+1})]\cdot 
\end{equation}

$$
\widehat{\Phi_r}(\xi+\alpha_2\xi_2 + ... + \alpha_d\xi_d ) d \xi d\xi_1 ... d\xi_{d+1}. 
$$
We need now the following

\begin{lemma}
If $F, F_1, ... , F_{d+1}$ and $\Phi$ are Schwartz functions, then one has

\begin{equation}\label{ave1}
\int_{\xi+\xi_1 + ... +\xi_d +\xi_{d+1} = 0}
\widehat{F}(\xi)
\widehat{F_1}(\xi_1) ...
\widehat{F_{d+1}}(\xi_{d+1})
\widehat{\Phi}(a \xi + a_1\xi_1 + ... + a_{d+1}\xi_{d+1}) d\xi d\xi_1 ... d\xi_{d+1} =
\end{equation}

$$
\int_{\R^2}
F(x-a t)
F_1(x - a_1 t) ...
F_{d+1}(x - a_{d+1} t) \Phi(t) d t d x,
$$
for every real numbers $a, a_1, ..., a_{d+1}$.
\end{lemma}

\begin{proof}
The formula is based on the following fact. If $\Gamma$ is a vector subspace of $\R^{d+2}$ and $\delta_{\Gamma}$ represents the Dirac distribution associated to it and defined by

$$\delta_{\Gamma} ( \phi ) = \int_{\Gamma} \phi(\gamma) d \gamma,$$
for every Schwartz function $\phi$ then, $\widehat{\delta_{\Gamma}} = \delta_{\Gamma^{\perp}}.$ In our case

\begin{equation}\label{gamma}
\Gamma = \{ (\xi, \xi_1, ..., \xi_{d+1}) \in \R^{d+2} : \xi+\xi_1 + ... +\xi_d +\xi_{d+1} = 0 \},
\end{equation}
and as a consequence $\Gamma^{\perp}$ is the one dimensional subspace along the vector $(1, 1, ..., 1)$. Using this and Plancherel, the left hand side of (\ref{ave1}) can be written as

$$\int_{\R}
\widehat{F}(\xi)
\widehat{F_1}(\xi_1) ...
\widehat{F_{d+1}}(\xi_{d+1})
\widehat{\Phi}(a \xi + a_1\xi_1 + ... + a_{d+1}\xi_{d+1})
e^{2\pi i x (\xi+\xi_1 + ... +\xi_{d+1})}  d\xi d\xi_1 ... d\xi_{d+1} d x.
$$
If one adds to it 

$$\widehat{\Phi}(a \xi + a_1\xi_1 + ... + a_{d+1}\xi_{d+1}) = \int_{\R} \Phi(t) e^{ - 2 \pi i t (a \xi + a_1\xi_1 + ... + a_{d+1}\xi_{d+1} )} d t , 
$$
one immediately obtains (\ref{ave1}), by using Fourier's inversion formula several times.
\end{proof}

We record also the following generalization of (\ref{ave1}), which will be used later on as well

\begin{equation}\label{ave2}
\int_{\Gamma}
\widehat{F}(\xi)
\widehat{F_1}(\xi_1) ...
\widehat{F_{d+1}}(\xi_{d+1})
\widehat{\Phi_1}(a \xi + a_1\xi_1 + ... + a_{d+1}\xi_{d+1})
\widehat{\Phi_2}(b \xi + b_1\xi_1 + ... + b_{d+1}\xi_{d+1}) d\xi d\xi_1 ... d\xi_{d+1}=
\end{equation}

$$
\int_{\R^3}
F(x-a t - b s)
F_1(x - a_1 t - b_1 s) ...
F_{d+1}(x - a_{d+1} t - a_{d+1} s ) \Phi_1 (t) \Phi_2 (s) d t d s d x.
$$
Now, if $G$ is an arbitrary Schwartz function and $a $ a real number, we denote by $G^a$ the function defined by

\begin{equation}\label{a}
\widehat{G^a}(\xi) = \widehat{G}(\xi) e^{ 2 \pi i a \xi}.
\end{equation}
Equivalently, one also has $G^a(x) = G(x+a)$.

Using this notation and applying (\ref{ave1}), our previous (\ref{24}) becomes

$$\int_{\R^2}
(f\ast\Phi_0^{0, \frac{n}{2^r}})(x-t)
(f_1\ast\Psi^{1, \frac{n_1}{2^r}}_r)(x)
\prod_{j=2}^d (f_j\ast\Phi_0^{j, \frac{n\alpha_j}{2^r}})(x-\alpha_j t)
(f_{d+1}\ast\Phi^{d+1}_0)(x) \Phi_r(t) d t d x =
$$

$$
\int_{\R^2}
(f\ast\Phi_0^{0, \frac{n}{2^r}})(x-t/2^r)
(f_1\ast\Psi^{1, \frac{n_1}{2^r}}_r)(x)
\prod_{j=2}^d(f_j\ast\Phi_0^{j, \frac{n\alpha_j}{2^r}})(x-\alpha_j t/2^r)
(f_{d+1}\ast\Phi^{d+1}_0)(x) \Phi_0(t) d t d x =
$$

\begin{equation}\label{25}
\int_{\R^2}
(f\ast\Phi_0^{0, \frac{n - t}{2^r}})(x)
(f_1\ast\Psi^{1, \frac{n_1}{2^r}}_r)(x)
\prod_{j=2}^d(f_j\ast\Phi_0^{j, \frac{(n - t)\alpha_j}{2^r}})(x)
(f_{d+1}\ast\Phi^{d+1}_0)(x) \Phi_0(t) d t d x. 
\end{equation}
If one performs a similar decomposition for an arbitrary scale $k\neq 0$ now, the analogous formula of (\ref{25}) becomes

\begin{equation}\label{26}
\int_{\R^2}
(f\ast\Phi_k^{0, \frac{n - t}{2^{r+k}}})(x)
(f_1\ast\Psi^{1, \frac{n_1}{2^{r+k}}}_{r+k})(x)
\prod_{j=2}^d(f_j\ast\Phi_k^{j, \frac{(n - t)\alpha_j}{2^{r+k}}})(x)
(f_{d+1}\ast\Phi^{d+1}_k)(x) \Phi_0(t) d t d x. 
\end{equation}
Summarizing everything, if one denotes by $\vec{\alpha} = (\alpha_2, ..., \alpha_d)$, one sees that the piece of $C_d$ that corresponds to Case $1_a'$, can be written as

\begin{equation}\label{split1}
\int_{[0,1]^{d-1}} \int_{\R}\left(
\sum_{r\leq 0} 2^r \sum_{n, n_1} C^r_{n, n_1} C_d^{r, n, n_1, \vec{\alpha}, t}\right) \Phi_0(t) d t d \vec{\alpha},
\end{equation}
where naturally $C_d^{r, n, n_1, \vec{\alpha}, t}$ is the operator whose $(d+2)$-linear form is given by the sum over $k$ of the corresponding inner expressions in (\ref{26}).

Clearly, in order to prove (\ref{6}) for (\ref{split1}), one would need to prove it for $C_d^{r, n, n_1, \vec{\alpha}, t}$ with upper bounds that are summable over $r, n, n_1$ and integrable
over $t$ and $\vec{\alpha}$. These operators $C_d^{r, n, n_1, \vec{\alpha}, t}$ are essentially paraproducts and for them one can apply the argument described in the previous section.
However, the presence of all of these parameters, have the role to {\it shift its ingredient functions} a little bit, so this time one has to be very precise when evaluates the size of the 
boundedness constants. As before, the idea is to apply the perfect corresponding (\ref{unu}) to all the indices $2\leq j\leq d$ for which $p_j=\infty$. 
This is possible because of the non-compact Littlewood-Paley $L^1$ normalized decompositions that have been used.
Denote by $S$ the set of indices
$2\leq j\leq d$ for which $p_j\neq \infty$. Now, if $l = |S| + 2$ and one freezes as before the $L^{\infty}$ normalized Schwartz functions corresponding to the indices in $\{2, ..., d\} \setminus S$, 
one obtains a {\it minimal} $l$-linear operator denoted by $C_d^{l, r, n, n_1, \vec{\alpha}, t}$.

\subsection*{Banach estimates for $C_d^{l, r, n, n_1, \vec{\alpha}, t}$}

Fix indices $1 < s_1, ..., s_{l+1} <\infty$ so that $1/s_1 + ... + 1/s_l = 1/s_{l+1}$. As in the previous section, the boundedness constants for

\begin{equation}\label{B}
C_d^{l, r, n, n_1, \vec{\alpha}, t} : L^{s_1}\times ... \times L^{s_l} \rightarrow L^{s_{l+1}}
\end{equation}
depend on the boundedness constants of the following two square functions

$$\left(\sum_k \left|f\ast\Phi_k^{0, \frac{n - t}{2^{r+k}}}(x)\right|^2\right)^{1/2} \,\, \text{and}\,\,
\left(\sum_k \left|f_1\ast\Psi^{1, \frac{n_1}{2^{r+k}}}_{r+k}(x)\right|^2\right)^{1/2}$$
and of several maximal functions of type

$$\sup_k \left|f_j\ast\Phi_k^{j, \frac{(n - t)\alpha_j}{2^{r+k}}}(x)\right|,$$
for $j\in S$. 

It is not difficult to see that the square functions are the continuous analogue of the shifted discrete square functions $S^{[\frac{n-t}{2^r}]}$ and 
$S^{n_1}$ of \cite{camil} and as a consequence, they are bounded on every $L^q$ space for $1< q <\infty$, with upper bounds of type $O(\log <\left[\frac{n-t}{2^r}\right] >)$ and
$O(\log <n_1>)$ respectively, see \cite{camil} \footnote{ For every real number $\gamma$, we denote by $[\gamma]$ its integer part.}.

Likewise, the maximal functions are bounded by the shifted maximal functions $M^{[\frac{(n-t)\alpha_j}{2^r}]}$ and therefore bounded on every $L^q$ space for $1< q <\infty$,
with bounds of type $O(\log <\left[\frac{(n-t)\alpha_j}{2^r}\right] >)$, see again \cite{camil}.

Using all these facts, one sees that the boundedness constants of (\ref{B}) are no greater than

\begin{equation}\label{Banach}
 C <r>^l (\log <n>)^l (\log <n_1>)^l (\log <[t]>)^l.
\end{equation}

\subsection*{Quasi-Banach estimates for $C_d^{l, r, n, n_1, \vec{\alpha}, t}$}

Fix indices $1< r_1, ..., r_l <\infty$ and $0<r_{l+1} <\infty$ so that $1/r_1+ ... 1/r_l = 1/r_{l+1}$. We would like to estimate this time the
boundedness constants of

\begin{equation}\label{q-B}
C_d^{l, r, n, n_1, \vec{\alpha}, t} : L^{r_1}\times ... \times L^{r_l} \rightarrow L^{r_{l+1}}
\end{equation}
and its adjoint operators.  To achieve this, we will have to discretize the corresponding (\ref{26}) even further (with respect to the $x$ variable), 
to be able to rewrite the operator $C_d^{l, r, n, n_1, \vec{\alpha}, t}$
in a form similar to (\ref{12}), for which one can apply (\ref{q-b}). One has first to observe that the bump functions corresponding to the index $1$ in (\ref{26}) are adapted to scales which
are $2^{-r}$ times greater than the scales of the bump functions corresponding to the other indices. This fact suggests that the natural thing to do is to discretize using the bigger scale.
On the other hand, one also observes that if a generic function $\Phi$ is a bump adapted to the dyadic interval $J$, and if $J\subseteq \widetilde{J}$ is another dyadic interval
$2^{-r}$ times greater that $J$, then $2^{5r}\Phi$ is a bump adapted to $\widetilde{J}$ as well ($5$ corresponds to the number of derivatives in the definition of {\it adaptedness}).

These facts, together with standard averaging and approximation arguments of \cite{camil} (including Fatou's lemma, etc.) show that our problem can be reduced to estimating expressions of type

\begin{equation}
\frac{1}{2^{6 rl}}
\sum_I \frac{1}{|I|^{(l-1)/2}}
|\langle f ,\Phi^0_{I_{[n-t]}} \rangle|
|\langle f_1, \Phi^1_{I_{n_1}}\rangle|
\prod_{j\in S}
|\langle f_j, \Phi^j_{I_{[(n-t)\alpha_j]}}\rangle|
|\langle f_{d+1}, \Phi^{d+1}_I\rangle|
\end{equation}
where the functions $f$,  $(f_j)_j$ are as in (\ref{q-b}) and $(p_j)_j$ there are the same as our $(r_j)_j$ here \footnote{ The power $6$ above should be read as $5+1$, where $5$ comes from the 
{\it adaptedness} argument and $1$ is a consequence of scaling,
in particular all the bump functions are $L^2$ normalized now.}.  Using (\ref{q-b}) and interpolation, we deduce that the boundedness constants of (\ref{q-B})
are no greater than

\begin{equation}\label{q-Banach}
2^{-6 l  r} (\log <n>)^l (\log <n_1>)^l (\log <[t]>)^l,
\end{equation}
and the same is true for all the adjoints of the operator.

\subsection*{The final interpolation}

Fix now indices $p$, $(p_j)_j$ as in (\ref{6}). Given that the desired estimates are {\it on the edge} of the {\it Banach region}, one can first use convexity arguments and linear interpolation only, to obtain many quasi-Banach
estimates whose bounds do not grow too much with respect to $r$ (at a rate of at most $2^{- \epsilon r}$ say, for some small $\epsilon$). Then, one can use the multilinear interpolation theory from \cite{mtt:general}
and interpolate between these better quasi-Banach estimates and the previous Banach ones in (\ref{Banach}), to realize that (\ref{6}) for $C_d^{r, n, n_1, \vec{\alpha}, t}$ comes with a bound which is acceptable by (\ref{split1}).

This completes the discussion of Case $1'_a$.

The rest of the cases follow a similar strategy. As one could observe, besides the {\it quadratic/logarithmic} argument, the presence of the decaying factor $2^r$ in (\ref{split1}) was also crucial. 
In the remaining of the paper we shall describe the adjustments that one sometimes needs to make in the other cases, in order for the above argument to work.

\subsection*{Case $1_a''$}

The $1''_a$ part, corresponding to $r >0$ is actually simpler, since this time the interaction between $\widehat{\Psi_r}(\xi_1)$ and $\widehat{\Phi^1_0}(\xi_1)$ gives

$$\widehat{\Psi_r}(\xi_1)\widehat{\Phi^1_0}(\xi_1) = \frac{1}{2^{rM}}\widehat{\widetilde{\Psi_r}}(\xi_1)$$
for some large constant $M>0$, where $\widehat{\widetilde{\Psi_r}}(\xi_1)$ is another $\Psi$ function adapted to the same scale as $\widehat{\Psi_r}(\xi_1)$.
This huge decaying factor together with a similar argument as before, solve this case as well.

\subsection*{Case $1_b$.}

This is very similar to $1_a$. In fact, the only difference is that this time the corresponding Fourier coefficients can be estimated by

$$|C_{n,n_1}|\lesssim \frac{1}{<n>^2} \frac{1}{<n-n_1>^{\#}} + \frac{1}{<n>^{\#}} \frac{1}{<n_1>^{\#}}, $$
as shown in \cite{camil} and this still gives a contribution summable over $n, n_1$. 

\subsection*{Case $1_c$.}

Here, one has first to realize that on the support of $\widehat{\Psi_r}(\widetilde{\xi})\widehat{\Phi_r}(\xi_1)$ the symbol $\int_0^1 1_{\R^{+}}(\widetilde{\xi} + \alpha_1\xi_1) d \alpha_1$
behaves like a classical Marcinkiewicz-H\"{o}rmander-Mihlin symbol and as a consequence, one has perfect decay for its Fourier coefficients, of type $\frac{1}{<n>^{\#}}\frac{1}{<n_1>^{\#}}$. 

There are two subcases $1_c'$ and $1_c''$,                                           
which correspond as before to $r< 0$ and $r\geq 0$ respectively.

\subsection*{Case $1_c'$}

In this situation, one just has to observe that

$$\widehat{\Phi_r}(\xi_1) \widehat{\Phi^1_0}(\xi_1) = \widehat{\Phi_r}(\xi_1) \xi_1 \widehat{\phi^1_0}(\xi_1) =
\widehat{\widetilde{\Phi_r}}(\xi_1) \xi_1 = 2^r \widehat{\widetilde{\Phi_r}}(\xi_1)\frac{\xi_1}{2^r} = 2^r \widehat{\widetilde{\Psi_r}}(\xi_1)
$$
where $\widehat{\widetilde{\Psi_r}}(\xi_1)$ is also of $\Psi$ type. The presence of the factor $2^r$ above, shows that this case can be treated exactly as the previous $1'_a$.

\subsection*{Case $1_c''$}

This time one observes that when the two functions  $\widehat{\Phi_r}(\xi_1)$ and $\widehat{\Phi^1_0}(\xi_1)$ interact, there is no decaying factor coming out of this and all one can say is that

$$\widehat{\Phi_r}(\xi_1) \widehat{\Phi^1_0}(\xi_1) = \widehat{\widetilde{\Phi^1_0}}(\xi_1)$$
where $\widehat{\widetilde{\Phi^1_0}}(\xi_1)$ is another $\Psi$ function,
since $\widehat{\Phi^1_0}(\xi_1)$ is adapted on an interval which lies inside the one corresponding to $\widehat{\Phi_r}(\xi_1)$ (recall that $r\geq 0$ now).
To produce a decaying factor, one would have to argue somewhat differently.

We will explain the changes that one has to make again in the $k=0$ case for simplicity, since as usual the argument is scale invariant. Consider the term

\begin{equation}\label{29}
\widehat{\Phi^0_0}(\xi)
\widehat{\Phi^1_0}(\xi_1) ...
\widehat{\Phi^{d+1}_0}(\xi_{d+1}) .
\end{equation}
Recall that $\widehat{\Phi^0_0}(\xi)$ is of $\Psi$ type, but it has also compact support, since the Littlewood-Paley decompositions have been chosen to be compact for the
$0$ and $d+1$ positions. Pick $\widehat{\widetilde{\Phi^0_0}}(\xi)$ another $\Phi$ function, supported on a slightly larger interval and which is equal to $1$
on the support of $\widehat{\Phi^0_0}(\xi)$. Then, split (\ref{29}) as

\begin{equation}\label{30}
\widehat{\Phi^0_0}(\xi)
\widehat{\Phi^1_0}(\xi_1) ...
\widehat{\Phi^{d+1}_0}(\xi_{d+1})\widehat{\widetilde{\Phi^0_0}}(\widetilde{\xi}) +
\end{equation}

\begin{equation}\label{31}
\widehat{\Phi^0_0}(\xi)
\widehat{\Phi^1_0}(\xi_1) ...
\widehat{\Phi^{d+1}_0}(\xi_{d+1})[1 - \widehat{\widetilde{\Phi^0_0}}(\widetilde{\xi})]. 
\end{equation}
The $(d+2)$- linear form determined by (\ref{30}) (and its family of analogs for each scale) can be treated as in the $1''_a$ case and in fact it is even simpler, since this time one has that
$\widehat{\Psi_r}(\widetilde{\xi})\widehat{\widetilde{\Phi^0_0}}(\widetilde{\xi}) = 0$ unless $r = 1, 2, 3$ (say).

To understand (\ref{31}), we rewrite it as

\begin{equation}\label{32}
\widehat{\Phi^0_0}(\xi)
\widehat{\Phi^1_0}(\xi_1) ...
\widehat{\Phi^{d+1}_0}(\xi_{d+1})[\widehat{\widetilde{\Phi^0_0}}(\xi) - \widehat{\widetilde{\Phi^0_0}}(\widetilde{\xi})] =
\end{equation}

$$
- \widehat{\Phi^0_0}(\xi)
\widehat{\Phi^1_0}(\xi_1) ...
\widehat{\Phi^{d+1}_0}(\xi_{d+1})\left[ \int_0^1 \widehat{\widetilde{\Phi^0_0}}'((1-s)\xi + s \widetilde{\xi}) d s \right] (\widetilde{\xi} - \xi) =
$$

$$
- \widehat{\Phi^0_0}(\xi)
\widehat{\Phi^1_0}(\xi_1) ...
\widehat{\Phi^{d+1}_0}(\xi_{d+1})\left[ \int_0^1 \widehat{\widetilde{\Phi^0_0}}'(\xi + s (\alpha_2\xi_2 + ... + \alpha_d \xi_d)) d s \right](\alpha_2\xi_2 + ... + \alpha_d \xi_d ).
$$
At this point we have to realize that we will loose another factor of $d$, because of the paranthesis $(\alpha_2\xi_2 + ... + \alpha_d \xi_d )$. Each of these $O(d)$ expressions
is of the form

\begin{equation}\label{33}
\widehat{\Phi^0_0}(\xi)
\widehat{\Phi^1_0}(\xi_1) ...
\widehat{\Phi^{d+1}_0}(\xi_{d+1})\left[ \int_0^1 \widehat{\widetilde{\Phi^0_0}}'(\xi + s (\alpha_2\xi_2 + ... + \alpha_d \xi_d)) d s \right]
\end{equation}
where, for some $2\leq j \leq d$ one has an extra factor of type $\alpha_j\xi_j$ in addition to the previous $\widehat{\Phi^j_0}(\xi_j)$. Clearly, this just adds another
harmless $\Psi$ function to that $j$ term, so it is enough to analize (\ref{33}). The crucial observation here is to realize that when one hits (\ref{33}) with a factor of type
$\widehat{\Psi_r}(\widetilde{\xi}) \widehat{\Phi_r}(\xi_1)$, one has to have $0\leq s \leq C/2^r$ in order for the corresponding term to be non-zero. This means that one can simply replace the integral 

$$
\int_0^1 \widehat{\widetilde{\Phi^0_0}}'(\xi + s (\alpha_2\xi_2 + ... + \alpha_d \xi_d)) d s 
$$
in (\ref{33}) by

$$
\int_0^{C/2^r} \widehat{\widetilde{\Phi^0_0}}'(\xi + s (\alpha_2\xi_2 + ... + \alpha_d \xi_d)) d s.
$$
Now, for each fixed $r > 0$ and each $0\leq s \leq C/2^r$, the corresponding form can be estimated exactly as before uniformly in $s$ and in the end, after integration, one obtains an upper bound 
summable over $r > 0$. The extra factor $\widehat{\widetilde{\Phi^0_0}}'(\xi + s (\alpha_2\xi_2 + ... + \alpha_d \xi_d))$ is of course harmless, since it just adds another average to the 
generic formula, as can be seen from (\ref{ave2}).

This ends Case $1$.

\subsection*{Case 2: $j_1=0$ and $j_2= d+1$}

The goal here is to show that after some calculations, one can in fact reduce this case to the previous Case $1$. To understand this, consider again a generic $k=0$ term, as the one in
(\ref{29}). One can split it as 

$$
\widehat{\Phi^0_0}(\xi)
\widehat{\phi^1_0}(\xi_1) ...
\widehat{\Phi^{d+1}_0}(\xi_{d+1}) + 
$$

$$
\widehat{\Phi^0_0}(\xi)
\widehat{\psi^1_0}(\xi_1) ...
\widehat{\Phi^{d+1}_0}(\xi_{d+1}) = A + B
$$
where $\widehat{\phi^1_0}(\xi_1)$ is of $\Phi$ type, compactly supported at scale one, while $\widehat{\psi^1_0}(\xi_1)$ is of $\Psi$ type also adapted at scale one. Clearly, the $B$ terms generate 
$(d+2)$- linear forms similar to the ones in Case $1$, and so it
is enough to discuss the $A$ terms only. Here it should not be difficult to realize that by construction, at least one of the two $\Psi$ functions $\widehat{\Phi^0_0}(\xi)$ and $\widehat{\Phi^{d+1}_0}(\xi_{d+1})$, has its
support away from zero. And moreover, we claim that without loss of generality, one can assume that $\widehat{\Phi^{0}_0}(\xi)$ is that function. To see this, one just has to observe that the roles of the variables
$\xi$ and $\xi_{d+1}$ are totally symmetric. Indeed, since $\xi+\xi_1+ ... + \xi_{d+1} = 0$ a simple change of variables shows that

$$\int_{[0,1]^d} 1_{\R_+}(\xi + \alpha_1\xi_1 + ... +\alpha_d \xi_d) d \alpha_1 ... d\alpha_d = \int_{[0,1]^d} 1_{\R_{-}}(\xi_{d+1} + \beta_1\xi_1 + ... +\beta_d \xi_d) d \beta_1 ... d\beta_d$$
which is obviously a similar symbol.

In particular, one can clearly rewrite $A$ as

\begin{equation}\label{34}
\widehat{\Phi^0_0}(\xi)
\widehat{\phi^1_0}(\xi_1) ...
\widehat{\Phi^{d+1}_0}(\xi_{d+1})\widehat{\widetilde{\Psi^{0}_0}}(\xi)
\end{equation}
for another well chosen compactly supported $\Psi$ function $\widehat{\widetilde{\Psi^{0}_0}}$. Then, one rewrites (\ref{34}) further as

$$
\widehat{\Phi^0_0}(\xi)
\widehat{\phi^1_0}(\xi_1) ...
\widehat{\Phi^{d+1}_0}(\xi_{d+1})\widehat{\widetilde{\Psi^{0}_0}}(\widetilde{\xi}) + 
$$

\begin{equation}\label{35}
\widehat{\Phi^0_0}(\xi)
\widehat{\phi^1_0}(\xi_1) ...
\widehat{\Phi^{d+1}_0}(\xi_{d+1})[ \widehat{\widetilde{\Psi^{0}_0}}(\xi ) -   \widehat{\widetilde{\Psi^{0}_0}}(\widetilde{\xi})].
\end{equation}
Then, one can remark that the symbol

$$\int_0^1 1_{\R^{+}}(\widetilde{\xi} + \alpha_1\xi_1) d \alpha_1$$
is a classical symbol on the support of the first term in (\ref{35}) and its analysis becomes simpler. 
In particular, one no longer needs to insert the extra decomposition over $r$ to study it.

We are thus left with the second term of (\ref{35}). Modulo a minus sign, this term can be written as

$$
\widehat{\Phi^0_0}(\xi)
\widehat{\phi^1_0}(\xi_1) ...
\widehat{\Phi^{d+1}_0}(\xi_{d+1})\left[\int_0^1 \widehat{\widetilde{\Psi^{0}_0}}'( (1-s)\xi + s \widetilde{\xi} ) d s \right]\cdot
$$

$$
\cdot (\alpha_2\xi_2 + ... + \alpha_d \xi_d ) = 
$$

$$
\widehat{\Phi^0_0}(\xi)
\widehat{\phi^1_0}(\xi_1) ...
\widehat{\Phi^{d+1}_0}(\xi_{d+1})\left[\int_0^1 \widehat{\widetilde{\Psi^{0}_0}}'( \xi + s ( \alpha_2 \xi_2 + ... + \alpha_d \xi_d)) d s \right]\cdot
$$

$$
\cdot (\alpha_2\xi_2 + ... + \alpha_d \xi_d ).
$$
Then, we decompose this last term further as

\begin{equation}\label{36}
\widehat{\Phi^0_0}(\xi)
\widehat{\phi^1_0}(\xi_1) ...
\widehat{\Phi^{d+1}_0}(\xi_{d+1})\left[ \int_0^1 \widehat{\widetilde{\Psi^{0}_0}}'( \xi + s ( -\xi)) d s \right]\cdot
\end{equation}

$$
\cdot (\alpha_2\xi_2 + ... + \alpha_d \xi_d ) +
$$

\begin{equation}\label{37}
\widehat{\Phi^0_0}(\xi)
\widehat{\phi^1_0}(\xi_1) ...
\widehat{\Phi^{d+1}_0}(\xi_{d+1})\left[ \int_0^1 \int_0^1 \widehat{\widetilde{\Psi^{0}_0}}''( (1-ts)\xi + (1-t) s ( \alpha_2 \xi_2 + ... + \alpha_d \xi_d) )  s d s d t \right]\cdot
\end{equation}

$$
\cdot (\alpha_2\xi_2 + ... + \alpha_d \xi_d ) \cdot \widetilde{\xi}.
$$
We claim now that the term in (\ref{36}) can be reduced to the Case 1 studied earlier. Indeed, notice first that the expression $(\alpha_2\xi_2 + ... + \alpha_d \xi_d)$
contributes $d-1$ terms and consider for instance the one corresponding to $\alpha_2\xi_2$. The presence of $\xi_2$ transforms the bump function depending on
this variable, into one
of $\Psi$ type in an intermediate position (as in Case 1) and then, one just has to observe that the {\it localized Fourier coefficients} of symbols of type

$$
(\widetilde{\widetilde{\xi}}, \xi_2) \rightarrow \int_0^1 \alpha_2 1_{\R_{+}}(\widetilde{\widetilde{\xi}} + \alpha_2 \xi_2) d \alpha_2 
$$
still satisfy the same crucial quadratic estimates that have been proved earlier in \cite{camil}. 

We are then left with the study of the terms in (\ref{37}). From now on (as many times before) we think of the variables $\alpha_2, ..., \alpha_d$ as being
{\it freezed} and of our symbol as being of type 

\begin{equation}\label{simbol}
\int_0^1 1_{\R^{+}}(\widetilde{\xi} + \alpha_1\xi_1) d \alpha_1.
\end{equation}
The variables $\xi_1$ and $\widetilde{\xi}$ are of course {\it special}, but so is $(\alpha_2\xi_2 + ... + \alpha_d \xi_d)$ as it appears quite explicitly in (\ref{37}).
Consider now an extra {\it paraproduct decomposition} of the identity, in the form of finitely many expressions of the type

\begin{equation}\label{pardec}
\sum_r \widehat{\Phi_r} (\xi_1)\widehat{\Phi_r}(\widetilde{\xi})\widehat{\Phi_r} (\alpha_2\xi_2 + ... + \alpha_d \xi_d).
\end{equation}
This can be easily obtained by combining three independent Littlewood-Paley decompositions. It is important to emphasize that at least one
of the above ingredient family of functions must be of $\Psi$ type. As in Case 1, the idea now is to insert this extra decomposition (\ref{pardec}) into (\ref{37}) and study the newly
formed expressions. The support of the function of two variables

$$(\xi_1, \widetilde{\xi}) \rightarrow  \widehat{\Phi_r} (\xi_1) \widehat{\Phi_r} (\widetilde{\xi}) $$
will clearly play an important role, since as long as it is a Whitney square, on it one can decompose the symbol (\ref{simbol}) as a double Fourier series,
precisely as in Case 1. We therefore witness two distinct situations.

\subsection*{The Whitney case}

In this case, the above supports are all Whitney squares with respect to the origin. This means that either $\widehat{\Phi_r}(\xi_1)$ or $\widehat{\Phi_r}(\widetilde{\xi})$
is of $\Psi$ type. It is also useful to observe that since $\xi = \widetilde{\xi} - (\alpha_2\xi_2 + ... + \alpha_d \xi_d)$, it must belong to an interval of size
$2^r$ centered at the origin.  But this means that one must have $r\geq 0$ in order for (\ref{pardec}) to have a nontrivial interaction with (\ref{37}) (remember that
from the beginning, we are in the case when $\widehat{\Phi_0^0}(\xi)$ is of $\Psi$ type and also compactly supported away from the origin).

The case when $\widehat{\Phi_r}(\xi_1)$ is of $\Psi$ type is easier since when (\ref{pardec}) interacts with (\ref{37}) the only non-zero
terms are those corresponding to indices $r$ belonging to the finite set $\{0,1,2\}$. After that, one just applies the method of Case 1. Notice that because of the terms
$(\alpha_2\xi_2 + ... + \alpha_d \xi_d) \cdot \widetilde{\xi}$ in (\ref{37}), one will loose another factor of type $O(d^2)$ (after distributing the inner terms around) which is clearly acceptable.

The other case, when $\widehat{\Phi_r}(\widetilde{\xi})$ is of $\Psi$ type is more complicated, since all the scales $r\geq 0$ can contribute. However, in this case one observes
that when (\ref{pardec}) interacts with (\ref{37}) then one must have either $s$ or $1-t$ smaller than $C/2^{r/2}$ in (\ref{37}) (for a certain fixed but large constant $C>0$).
But then this shows that this case can be treated exactly as the previous Case $1''_c$.

\subsection*{The non-Whitney case}

This case corresponds to the situation when both $\widehat{\Phi_r}(\widetilde{\xi})$  and $\widehat{\Phi_r}(\xi_1)$ are of $\Phi$ type. However, as a consequence 
of (\ref{pardec}), the
support of  $\widehat{\Phi_r} (\alpha_2\xi_2 + ... + \alpha_d \xi_d)$ must be an interval of the same size $2^r$ whose distance to the origin is comparable to its length.
In particular, $\xi = \widetilde{\xi} - (\alpha_2\xi_2 + ... + \alpha_d \xi_d)$ must belong to an interval of a similar kind. But then, because of the presence of 
$\widehat{\Phi_0^0}(\xi)$, one must have $r\sim 0$ in order for (\ref{pardec}) to have a nontrivial interaction with (\ref{37}). So the new term that gets multiplied with
(\ref{37})  in this case, is of the form

\begin{equation}\label{scara0}
\widehat{\Phi_0} (\xi_1)\widehat{\Phi_0}(\widetilde{\xi})\widehat{\Phi_0} (\alpha_2\xi_2 + ... + \alpha_d \xi_d).
\end{equation}
At this point it is important to remember about the factor $\widetilde{\xi}$ in (\ref{37}). When it is multiplied with the above $\widehat{\Phi_0}(\widetilde{\xi})$
it transforms this function into one of a $\Psi$ type, which is clearly very good news. Let us denote it by $\widehat{\Psi_0}(\widetilde{\xi})$ for the rest of the discussion.
We are still not done yet, since this new $\Psi$ type function does not have a support
away from the origin.  However, we can apply to this situation a treatment similar to the one used in the previous Case 1. Before doing that, to summarize, the expression that we face now
consists of a product of a term of the type

\begin{equation}\label{scara00}
\widehat{\Phi_0} (\xi_1)\widehat{\Psi_0}(\widetilde{\xi})\widehat{\Phi_0} (\alpha_2\xi_2 + ... + \alpha_d \xi_d)
\end{equation}
with the previous (\ref{37}) which no longer contains the original factor $\widetilde{\xi}$. At this point, insert another decomposition of the identity, of the type

$$\sum_r \widehat{\Phi_r} (\xi_1)\widehat{\Phi_r}(\widetilde{\xi})$$
where as before either $\widehat{\Phi_r} (\xi_1)$ or $\widehat{\Phi_r}(\widetilde{\xi})$ are of $\Psi$ type. And finally, exactly as in Case 1, observe that when
this new decomposition gets multiplied with the above (\ref{scara00}), the index $r$ must be smaller than zero to obtain nontrivial terms, and also a small factor of the type
$2^r$ jumps out naturally, from the interaction between $\widehat{\Phi_r}(\widetilde{\xi})$ and $\widehat{\Psi_0}(\widetilde{\xi})$. After that, the argument is identical to the one
used before in Case 1.

\subsection*{Case 3: $j_1=2$ and $j_2=3$}

Finally, it is not difficult to see that the Case $3$ can be analyzed as Case $1$, since there are now two $\Psi$ type functions in intermediate positions.

This ends our proof since by symmetry, any other case can be reduced to one of these three.

\section{Generalizations}

To be able to describe and motivate the generalizations we mentioned at the beginning of the paper, we would first like to recall the classical calculations of Calder\'{o}n, which gave rise to his
commutators.

\subsection*{Calculus with functions of linear growth}

Let $A(x)$ be a complex valued function of one real variable having {\it linear growth}, more precisely satisfying $A'\in L^{\infty}$. We denote by $Hf(x)$ the classical Hilbert transform given by

\begin{equation}\label{hilbert}
Hf(x) : = p.v. \int_{\R} f(x-y) \frac{dy}{y}
\end{equation}
and by $Af(x)$ the operator of multiplication with $A(x)$. The question that started the whole theory, was whether the commutator $[H, A]$ was smoothing of order one, or equivalently if $[H,A]\circ D$ maps
$L^p$ into $L^p$ boundedly for $1<p<\infty$, where $Df(x):= f'(x)$\footnote{As a consequence of this, $HA$ could be written as $HA = AH + [H,A]$ and therefore belonged to Calder\'{o}n's algebra.}.

For $f$ smooth and compactly supported, one can write

$$(HA - AH)\circ D (f)(x)= (HA - AH)(f')(x) =$$

$$p.v. \int_{\R} A(x-y) f'(x-y) \frac{dy}{y} - p.v. \int_{\R} A(x) f'(x-y)\frac{dy}{y} =$$

$$p.v. \int_{\R} A(y) f'(y)\frac{1}{x-y} d y - p.v. \int_{\R} A(x) f'(y)\frac{1}{x-y} dy =$$

$$ - p.v. \int_{\R} \frac{A(x) - A(y)}{x-y} f'(y) d y = p.v. \int_{\R}\left( \frac{A(x) - A(y)}{x-y}  \right)' f(y) d y =$$

$$ - p.v. \int_{\R} \frac{A'(y)}{x-y} f(y) d y + p.v. \int_{\R} \frac{A(x) - A(y)}{(x-y)^2} f(y) d y =$$

$$ - H(A' f)(x) + C_1 f(x)$$
where $C_1$ is precisely the first Calder\'{o}n commutator. Since both $f\rightarrow H(A'f)$ and $C_1$ are bounded on $L^p$ for $1<p<\infty$, these show that indeed $[H,A]$ is smoothing of order one.

Besides the commutators, Calder\'{o}n pointed out that more general operators such as

\begin{equation}\label{classic1}
f\rightarrow p.v. \int_{\R} F \left( \frac{A(x) - A(y)}{x-y}  \right) f(y) \frac{dy}{x-y}
\end{equation}
or even

\begin{equation}\label{classic2}
f\rightarrow p.v. \int_{\R} F \left( \frac{A(x) - A(y)}{x-y}  \right)    G \left( \frac{B(x) - B(y)}{x-y}  \right)f(y) \frac{dy}{x-y}
\end{equation}
are worthwile to be studied as well, since they appear naturally in complex analysis or boundary value problems in PDE \footnote{Here, the functions $F, G$ are for instance analytic in a certain disc around the origin, while $\|A'\|_{\infty}$, $\|B'\|_{\infty}$ 
are supposed to be strictly smaller than the radii of convergence of $F$ and $G$ respectively. And of course, one can consider similar operators with more than two factors.  }.

The work of Coifman, McIntosh and Meyer \cite{cmm} proved the desired estimates for all these operators, by reducing them to the previous (\ref{3}).
Finally, let us also remark that the above calculations show that if one denotes by $a:= A'$, then for $g, Dg \in L^p$ with $1<p<\infty$, one has

\begin{equation}
H(A Dg) = A H(Dg) - H(a Dg) + C_1 g.
\end{equation}
This is remarkable since apriorily, there is no direct and obvious way to even define $H(A Dg)$. Notice that on the right hand side, all the compositions make sense.

\subsection*{Calculus with functions of polynomial growth, Part 1}

Suppose now that $A(x)$ is a function having {\it quadratic growth}, more precisely satisfying $A''\in L^{\infty}$. {\it Is it true that the commutator $[H,A]$ is still a smoothing operator} ?
We will see that this time $[H,A]$ is smoothing of order two. 

Indeed, for $f$ a smooth and compactly supported function, given also the previous calculations, one has

$$(HA - AH) \circ D^2 (f)(x) = (HA - AH) \circ D (f')(x) =$$

$$p.v. \int_{\R}\left(\frac{A(x)-A(y)}{(x-y)^2} - \frac{A'(y)}{x-y}\right) f'(y) d y =$$

$$p.v. \int_{\R} \frac{A(x) - A(y) - A'(y)(x-y)}{(x-y)^2} f'(y) d y =$$

$$ - p.v. \int_{\R}\frac{-A'(y) - A''(y)(x-y) + A'(y)}{(x-y)^2} f(y) d y + $$

$$2 p.v. \int_{\R} \frac{A(x) - A(y) - A'(y)(x-y)}{(x-y)^3} f(y) d y = $$

$$ H(A'' f)(x) + 2 p.v. \int_{\R}\left( \frac{A(x) - T_y^1 A(x)}{(x-y)^2}\right) \frac{f(y)}{x-y} d y$$
where

$$T_y^1 A(x) := A(y) + \frac{A'(y)}{1!} (x-y)$$
is the Taylor polynomial of order 1 of the function $A$, about the point $y$. Since $f\rightarrow H(A'' f)$ is clearly a bounded operator, the problem reduces to proving $L^p$ bounds for the linear operator

\begin{equation}\label{taylor1}
f \rightarrow p.v. \int_{\R}\left( \frac{A(x) - T_y^1 A(x)}{(x-y)^2}\right) \frac{f(y)}{x-y} d y.
\end{equation}
For functions of arbitrary {\it polynomial growth} satisfying $A^{(d)}\in L^{\infty}$ for some $d\geq 1$, one can similarly show that $[H,A]$ is smoothing of order $d$, if the operator 

\begin{equation}\label{taylor2}
f \rightarrow p.v. \int_{\R}\left( \frac{A(x) - T_y^{d-1} A(x)}{(x-y)^{d}}\right) \frac{f(y)}{x-y} d y
\end{equation}
satisfies the usual estimates, where $T_y^{d-1} A(x)$ is the Taylor polynomial of order $d-1$ of the function $A$, about the point $y$.

More generally, as before, one can ask if the operators

\begin{equation}\label{taylor3}
f\rightarrow p.v. \int_{\R} F \left( \frac{A(x) - T_y^{d-1} A(x)}{(x-y)^{d}}\right) \frac{f(y)}{x-y} d y
\end{equation}
or even

\begin{equation}\label{taylor4}
f\rightarrow p.v. \int_{\R} G \left( \frac{B(x) - T_y^{d_1-1} B(x)}{(x-y)^{d_1}}\right)  H \left( \frac{C(x) - T_y^{d_2-1} C(x)}{(x-y)^{d_2}}\right) \frac{f(y)}{x-y} d y
\end{equation}
are bounded on $L^p$ for $1<p<\infty$, assuming that $F,G,H$ are analytic as before, while $A^{(d)}, B^{(d_1)}, C^{(d_2)}$ are all in $L^{\infty}$ \footnote{And in fact, even more generally, one can consider
operators having an arbitrary number of similar factors.}. As a consequence of the method in this paper, one can answer all these questions affirmatively. Indeed, it is not difficult to see, thanks to the general averaging formula
for the rest of a Taylor series, that as before, the problem reduces to proving $L^p\times L^{\infty}\times ... \times L^{\infty} \rightarrow L^p$ estimates for the $(k+1)$-linear operator with symbol

$$\frac{1}{(d!)^k}\int_{[0,1]^k}\sgn (\xi +\alpha_1 \xi_1 + ... + \alpha_k \xi_k) (1-\alpha_1)^d ... (1-\alpha_k)^d d\alpha_1 ... d\alpha_k$$
which grow at most polynomially in $k$. But this symbol is very similar to the symbol of the $k$th Calder\'{o}n commutator and because of this, one can prove the desired estimates similarly. In fact, for $k=1$ one can actually see that

$$\int_0^1 \sgn(\xi + \alpha_1 \xi_1) (1-\alpha_1)^d d\alpha_1$$
is even better than the symbol of the first commutator, since at least along the line $\xi+\xi_1=0$ it becomes smoother, due to the presence of the extra factor $(1-\alpha_1)^d$. In particular, the {\it quadratic estimates} for its Fourier coefficients
are still available.

It is quite likely that the $T1$ theorem of David and Journ\'{e} \cite{dj} can be used to handle the cases when $F,G,H$ are of the form $x^n$ for some positive integer $n$. In fact, for $F(x) = x$ this has been verified in \cite{meyerc}
\footnote{It is also interesting to see in \cite{meyerc} at page 94, another instance where these operators appear naturally.}.
The more general case $G(x)=H(x)=x$ has also been treated directly in \cite{cohen}.

\subsection*{Calculus with functions of polynomial growth, Part 2}

There is an alternative calculation that one can perform to understand the previous question. Let us first observe that $[H,A]$ is smoothing of order two if and only if $[H,A]\circ H$ is smoothing of order two. Then, using the fact that
$H^2 = - I$, one can write

$$(AH-HA)\circ H (f) = AH^2 f - H(AH f) =$$

$$-A f - H(AH f) = \frac{1}{2}\left(H^2 (Af) + A H^2 f - 2 H( A H f)\right).$$
If one ignores the $\frac{1}{2}$ factor, the above paranthesis (calculated at an arbitrary point $x$) becomes

$$p.v. \int_{\R^2} f(x+t+s) A(x+t+s) \frac{dt}{t} \frac{ds}{s} + p.v. \int_{\R^2} f(x+t+s) A(x) \frac{dt}{t} \frac{ds}{s} - $$

$$p.v. \int_{\R^2} f(x+t+s) A(x+t) \frac{dt}{t} \frac{ds}{s} - p.v. \int_{\R^2} f(x+t+s) A(x+s) \frac{dt}{t} \frac{ds}{s} = $$

$$p.v. \int_{\R^2} f(x+t+s)\left( A(x+t+s) - A(x+t) - A(x+s) + A(x)\right) \frac{dt}{t} \frac{ds}{s}.$$
To see if the last expression is smoothing of order two, one has (after integrating by parts with respect to the $t$ variable)

$$p.v. \int_{\R^2} f''(x+t+s)\left( \frac{A(x+t+s) - A(x+t) - A(x+s) + A(x)}{ t s}\right) d t d s = $$

$$ - p.v. \int_{\R^2}f'(x+t+s) \left(\frac{A'(x+t+s) - A'(x+t)}{t s}\right) d t d s + $$

$$ p.v. \int_{\R^2} f'(x+t+s) \left(\frac{A(x+t+s) - A(x+t) - A(x+s) + A(x)}{t^2 s}\right) d t d s : =$$

$$I + II.$$
Integrating by parts with respect to the $s$ variable now, one can rewrite $I$ as

$$p.v. \int_{\R^2}f(x+t+s) A''(x+t+s) \frac{d t}{t} \frac{d s}{s} - $$

$$p.v. \int_{\R^2} f(x+t+s) \frac{A'(x+t+s) - A'(x+s)}{s^2} d s \frac{d t }{t} =$$

$$ - A''(x) f(x) - H(C_{1, A'} f)(x),$$
where $C_{1, A'}$ is the Calder\'{o}n first commutator associated to the Lipschitz function $A'$.

Similarly, one can rewrite $II$ as

$$ - H(C_{1, A'} f)(x) +  p.v. \int_{\R^2} f(x+t+s)\frac{A(x+t+s) - A(x+t) - A(x+s) + A(x)}{ t^2 s ^2} d t d s =  $$

$$ - H(C_{1, A'} f)(x) +     $$

\begin{equation}\label{new}
p.v. \int_{\R^2} f(x+t+s) \left(\frac{\Delta_t}{t}\circ \frac{\Delta_s}{s} A(x)\right) \frac{d t}{t} \frac{d s}{s},
\end{equation}
where in general, $\Delta_h g(x)$ denotes the usual finite difference at scale $h$ given by 

$$\Delta_h g(x) = g(x+h) - g(x).$$
The problem can be therefore reduced to the one of proving $L^p$ estimates for the linear operator in (\ref{new}). The reader may remember it from our previous paper \cite{camil}.
If $A$ has arbitrary {\it polynomial growth} (i.e. $A^{(d)}\in L^{\infty}$ for some $d\geq 1$), then an analogous calculation shows that $[H,A]$ is smoothing of order $d$, if the
operator

\begin{equation}\label{new1}
f\rightarrow p.v. \int_{\R^d} f(x+t_1+ ... + t_d) \left(\frac{\Delta_{t_1}}{t_1}\circ ... \circ \frac{\Delta_{t_d}}{t_d} A(x)\right) \frac{d t_1}{t_1} ... \frac{d t_d}{t_d}
\end{equation}
is $L^p$ bounded.

And then more generally, one can ask the same question about the operator given by the expression

\begin{equation}\label{new2}
p.v. \int_{\R^d} f(x+t_1+ ... + t_d) \,F \left(\frac{\Delta_{t_1}}{t_1}\circ ... \circ \frac{\Delta_{t_d}}{t_d} A(x)\right) \frac{d t_1}{t_1} ... \frac{d t_d}{t_d}
\end{equation}
or by (in the case of two factors)

\begin{equation}\label{new3}
p.v. \int_{\R^d} f(x+t_1+ ... + t_d) \,G \left(\frac{\Delta_{a_1 t_1}}{ t_1}\circ ... \circ \frac{\Delta_{a_d t_d}}{ t_d} B(x)\right) 
H \left(\frac{\Delta_{b_1 t_1}}{ t_1}\circ ... \circ \frac{\Delta_{b_d t_d}}{ t_d} C(x)\right)
\frac{d t_1}{t_1} ... \frac{d t_d}{t_d}
\end{equation}
where $(a_j)_j$, $(b_j)_j$ are all different from zero real numbers and $F,G,H$ analytic functions as before. It is also interesting to compare these formulae with the classical expression (\ref{classic1}), which can be rewritten as

$$
p.v. \int_{\R} f(x+t)  F \left(\frac{\Delta_{t}}{t}A(x)\right) \frac{d t}{t}.
$$
We now claim that essentially without any extra effort, one can prove the desired estimates for all these operators. Indeed, using that

$$
\frac{\Delta_{t_1}}{t_1}\circ ... \circ \frac{\Delta_{t_d}}{t_d} A(x) = \int_0^1 ... \int_0^1 A^{(d)}(x+\alpha_1 t_1 + ... +\alpha_d t_d) d\alpha_1 ... d\alpha_d ,
$$
it is not difficult to see that (\ref{new2}) can be reduced to the study of the $(k+1)$-linear operator with symbol \footnote{Notice that this is precisely the symbol of the $k$th Calder\'{o}n commutator raised to the power $d$ !}

\begin{equation}\label{symbol1}
\left(\int_{[0,1]^k} \sgn (\xi + \alpha_1 \xi_1 + ... + \alpha_k \xi_k) d \alpha_1 ... d \alpha_k\right)^d
\end{equation}
while the more general (\ref{new3}) to $(k+1)$-linear operators whose symbols are products of type

\begin{equation}\label{symbol2}
\prod_{i=1}^d \left(\int_{[0,1]^k} \sgn (\xi + c_1^i\alpha_1 \xi_1 + ... + c_k^i\alpha_k \xi_k) d \alpha_1 ... d \alpha_k\right)
\end{equation}
for various non-zero real numbers $(c_j^i)_{i,j}$. The only fact that needs to be realized at this point is that the method extends naturally to cover {\it product symbols} of type (\ref{symbol1}) and (\ref{symbol2}) as well, since each individual factor
can be decomposed as before, as a Fouries series with Fourier coefficients that decay at least quadratically.
More specifically, the only difference in the argument is that instead of the Littlewood Paley decomposition in (\ref{extraLP}), which works very well in the $d=1$ case, 
one has to consider a product of $d$ such similar decompositions, each naturally corresponding to the factors of (\ref{symbol2})\footnote{For instance, for $d=2$, one takes a product of two $2$-dimensional Littlewood Paley decompositions
of type (\ref{extraLP}), one for the pair of variables $(\xi_1, \widetilde{\xi})$ corresponding to the first factor and another one for the pair of variables $(\xi_1, \widetilde{\widetilde{\xi}})$ corresponding to the second factor.}.

\subsection*{Extended Calder\'{o}n algebra}

Given the previous discussions, it is also natural to consider operators of type

\begin{equation}\label{doi}
f\rightarrow p.v. \int_{\R} f(x+t) F \left(\frac{\Delta_t}{t}\circ \frac{\Delta_t}{t} A(x)\right) \frac{d t}{t}
\end{equation}
or even

\begin{equation}
f\rightarrow p.v. \int_{\R} f(x+t) G \left(\frac{\Delta_t}{t}B(x)\right)  H \left(\frac{\Delta_t}{t}\circ \frac{\Delta_t}{t} C(x)\right)\frac{d t}{t}
\end{equation}
and so on, and ask if they are $L^p$ bounded if $F, G, H$ are analytic and $A'', B', C''\in L^{\infty}$. They can also be treated by the method of this paper. For instance, the study of (\ref{doi})
can be reduced to the study of multilinear operators with symbols of type

$$\int_{[0,1]^{2k}} \sgn (\xi + (\alpha_1 + \beta_1) \xi_1 + ... + (\alpha_k + \beta_k) \xi_k ) d \alpha d \beta$$
which clearly can be analysed in a similar manner. {\it Are they related to Calder\'{o}n's algebra in any way ?} one might ask
\footnote{Generally speaking, we say that an operator belongs to Calder\'{o}n's algebra, if it can be written as
a sum between a classical Fourier integral operator with a well defined symbol $a(x,\xi)$ (such as $A H$ for example) and a smoothing of a certain order term.}.
To be able to answer this, we need to recall the bilinear Hilbert transform \cite {lt1}, \cite{lt2}.

If $\alpha$ is any real number with $\alpha \neq 0, 1$, the bilinear Hilbert transform with parameter $\alpha$ denoted by $BHT_{\alpha}$, is the bilinear operator defined by

\begin{equation}\label{bht}
BHT_{\alpha}(f, g)(x) := p.v. \int_{\R} f(x+t) g(x+\alpha t) \frac{d t}{t}.
\end{equation}
It is known that these operators satisfy many $L^p$ estimates of H\"{o}lder type \cite {lt1}, \cite{lt2}. As a consequence of the results in \cite {lt1}, \cite{lt2} and of the $L^p$ theorem for (\ref{doi}) in the particular case
$F(x) = x$, one obtains by a straightforward calculation that

$$BHT_2 (f, A) (x) - 2 H A f (x) + A H f(x)$$
is a smoothing of order two expression, if $A''\in L^{\infty}$. In particular, given also the earlier calculations with functions of quadratic growth, one has that

$$BHT_2 (f, A) = A H f + \text{smoothing of order $2$ term},$$
which shows that the operator $f\rightarrow BHT_2 (f, A)$ belongs to Calder\'{o}n's algebra \footnote{An even simpler calculation shows that if $A'\in L^{\infty}$, then $BHT_2 (f, A) = A H f$ plus a smoothing of order one term.}.
Also, the parameter $2$ above can be replaced by any other $\alpha$, as a consequence of the fact that the theorem for (\ref{doi}) holds true if one replaces $\frac{\Delta_t}{t}\circ \frac{\Delta_t}{t}$ by 
$\frac{\Delta_{at}}{t}\circ \frac{\Delta_{bt}}{t}$ for some appropriate constants $a, b$.

Finally, let us also remark that one can {\it add derivatives} freely to the general operators in (\ref{new2}) and (\ref{new3}) as well and still obtain bounded on $L^p$ operators. A typical example would be

\begin{equation}
f\rightarrow p.v. \int_{\R^2}
f(x+t+s)
F \left(\frac{\Delta_t}{t}\circ \frac{\Delta_t}{t}\circ \frac{\Delta_s}{s} A(x)\right)               
G \left(\frac{\Delta_t}{t}\circ \frac{\Delta_s}{s}\circ \frac{\Delta_s}{s} \circ  \frac{\Delta_s}{s} B(x)\right)         
\frac{d t}{t} \frac{d s}{s}
\end{equation}
for $A''', B''''\in L^{\infty}$.
The straightforward (by now) details are left to the reader \footnote{And the same is true if one considers the natural generalizations of them, when one may face more than two kernels and more than two factors in the 
corresponding expressions.}.

\subsection*{Circular commutators}

In \cite{camil} we noticed a certain {\it symmetric} bilinear operator which we named {\it circular commutator}. 

We would like to describe its natural trilinear generalization here
\footnote{There are of course many other multilinear generalizations of this, as the reader can imagine.}. Thus, we record the following theorem.

\begin{theorem}
Let $a\neq 0$, $b\neq 0$ and $c\neq 0$ be three fixed real numbers. Consider also three Lipschitz functions $A$, $B$ and $C$. Then, the following expression

$$
p.v. \int_{\R^3}
\left(\frac{\Delta_{a t_1}}{t_1} A(x+t_2)\right)
\left(\frac{\Delta_{b t_2}}{t_2} B(x+t_3)\right)
\left(\frac{\Delta_{c t_3}}{t_3} C(x+t_1)\right)\frac{ d t_1}{t_1}\frac{ d t_2}{t_2}\frac{ d t_3}{t_3}
$$
viewed as a trilinear operator in $A'$, $B'$ and $C'$ maps $L^{p_1}\times L^{p_2} \times L^{p_3}$ into $L^p$ boundedly,
for  every $1<p_1, p_2, p_3\leq \infty$ with $1/p_1+1/p_2+1/p_3 = 1/p$ and $1/2 < p < \infty$.
\end{theorem}
The proof of this theorem uses again the same method. If we assume for simplicity that $a=b=c=1$, then the symbol of the corresponding trilinear operator
is given by following {\it circular product}

$$m_2(\xi_1, \xi_2, \xi_3)\cdot  m_2(\xi_2, \xi_3, \xi_1) \cdot m_2(\xi_3, \xi_1, \xi_2)$$
and such symbols can be clearly treated in the same manner.

\subsection*{$T1$ calculations}

It is very well known that the $T1$ theorem \cite{dj} can handle Calder\'{o}n commutators quite successfully. More precisely, it allows one to reduce the study of $C_k$ to the study of $C_{k-1}$. 
Let us briefly recall the details, in the particular case of the first two commutators.

Assume that $A$ is smooth, compactly supported and $A'\in L^{\infty}$. Rewrite $C_{1,A}$ as

$$C_{1,A}f(x) = p.v. \int_{\R}f(x+t) \frac{A(x+t) - A(x)}{t^2} d t.$$
Then, integrating by parts, one gets

$$C_{1,A} 1 (x) = p.v. \int_{\R} \frac{A(x+t) - A(x)}{t^2} d t =  - p.v. \int_{\R} (A(x+t) - A(x))\left( \frac{1}{t}\right)' d t =$$

$$ p.v. \int_{\R} A'(x+t) \frac{d t}{t} = H(A')(x)$$
which is a $BMO$ function.

Suppose now that $A, B$ are smooth, compactly supported and $A', B'\in L^{\infty}$. As before, rewrite the {\it second commutator}  associated to $A, B$ as

$$C_{2, A, B}f(x) = p.v. \int_{\R} f(x+t) \frac{(A(x+t) - A(x)) (B(x+t) - B(x))}{t^3} d t.$$
Then, one has

$$C_{2, A, B} 1 (x) = p.v. \int_{\R} \frac{(A(x+t) - A(x)) (B(x+t) - B(x))}{t^3} d t = $$

$$-\frac{1}{2} p.v. \int_{\R} (A(x+t) - A(x)) (B(x+t) - B(x)) \left(\frac{1}{t^2}  \right)' d t =$$

$$\frac{1}{2} p.v. \int_{\R} A'(x+t) \frac{B(x+t) - B(x)}{t^2} d t +$$

$$\frac{1}{2} p.v. \int_{\R} B'(x+t) \frac{A(x+t) - A(x)}{t^2} d t =$$

$$\frac{1}{2}  C_{1, B}(A')(x) + \frac{1}{2} C_{1, A}(B')(x)$$
which are both $BMO$ functions. Similar calculations can be performed to more general {\it commutators} of type

\begin{equation}
f \rightarrow \int_{\R}
f(x+t)
\left(\frac{\Delta_{ a_1 t}}{t} A_1(x)\right) ...
\left(\frac{\Delta_{ a_d t}}{t} A_d(x)\right)
\frac{d t}{t}
\end{equation}
for every sequence $(a_j)_j$ of nonzero real numbers.

Even though the operators of the previous sections do not come in a {\it standard form} (in particular, their kernels have a product structure) it is still tempting
to check if there is an analogous $T1$ type  {\it reduction of complexity}, available for them.

Consider this time $A, B$ smooth and compactly supported functions, satisfying $A'', B'' \in L^{\infty}$. The simplest analogue of the first commutator, is the operator $T_{1,A}$ given by

$$T_{1,A}f(x) = p.v.\int_{\R^2} f(x+t+s) 
\left(\frac{\Delta_{a t}}{ t}\circ \frac{\Delta_{b t}}{ t} A(x) \right)
\frac{d t}{t} \frac{d s}{s}$$
for $a, b$ real and nonzero.
One can write

$$T_{1, A} 1(x) =  p.v. \int_{\R^2} \frac{A(x+a t+ b s) - A(x+ a t) - A(x+ b s) + A(x)}{t^2 s^2} d t ds =$$

$$   p.v.\int_{\R^2} (A(x+a t+ b s) - A(x+ a t) - A(x+ b s) + A(x)) \left( \frac{1}{t}\right)' \left( \frac{1}{s}\right)' d t d s =$$

$$p.v. \int_{\R^2} A''(x+ a t+ b s) \frac{1}{t} \frac{1}{s} d t d s = H\circ H(A'')(x) = -A''(x)$$
which is in $L^{\infty}$ and therefore still in $BMO$.

However, as we will see, the problem becomes more complicated and the symmetry gets broken at the next step, when one considers the analogue of the second commutator $T_{2, A, B}$ given by

\begin{equation}\label{t2}
T_{2, A, B} f(x) = p.v. \int_{\R^2} f(x+t+s) 
\left(\frac{\Delta_{a t}}{ t}\circ \frac{\Delta_{b s}}{ s} A(x)\right ) 
\left(\frac{\Delta_{c t}}{ t}\circ \frac{\Delta_{d s}}{ s} B(x)\right) \frac{d t}{t} \frac{d s}{s}
\end{equation}
for $a, b, c, d$ real and nonzero numbers.
This time, one has

$$T_{2, A, B} 1(x) = \frac{1}{4} p.v. \int_{\R^2}(\Delta_{a t}\circ \Delta_{b s} A(x)) (\Delta_{c t}\circ \Delta_{d s} B(x))
\left(\frac{1}{t}\right)''
\left(\frac{1}{s}\right)''
 d t d s .$$
One particular term that appears after integrating by parts, is the one which corresponds to the situation when a pair of $s$ and $t$ derivatives hits the first term and another similar pair hits the second term. The operator obtained in this way is

$$p.v. \int_{\R^2} A''(x+ a t + b s) B''(x + c t + d s) 
\frac{d t}{t}
\frac{d s}{s}
$$
which is clearly of a $BHT_{\alpha, \beta}$ type and it is unlikely that such operators map $L^{\infty}\times L^{\infty}$ into $BMO$. 

Even worse, if instead of (\ref{t2}) one considers its natural generalization with $4$ factors $T_{4, A_1, A_2, A_3, A_4}$, a similar calculation generates the expression

$$p.v. \int_{\R^2}
A_1''(x+ a_1t + b_1s)
A_2''(x+ a_2t + b_2s)
A_3''(x+ a_3t + b_3s)
A_4''(x+ a_4t + b_4s)
\frac{d t}{t}
\frac{d s}{s}
$$
and it is known that for generic choices of $(a_j)_j$ and $(b_j)_j$ this $4$-linear operator does not satisfy any $L^p$ estimates of H\"{o}lder type \cite{camil2}.

\end{document}